\documentclass[11pt]{amsart}

\usepackage{amsmath, amssymb, amscd, cancel, color, graphicx}

\usepackage[all]{xy}
\usepackage{multirow}
\usepackage{longtable}
\usepackage{array}
\usepackage{paralist}
\usepackage{pinlabel}
\usepackage{enumitem}
\usepackage{bm}
\usepackage{hyperref}

\newtheorem{theorem}{Theorem}[section]

\newtheorem{corollary}[theorem]{Corollary}
\newtheorem{lemma}[theorem]{Lemma}
\newtheorem{proposition}[theorem]{Proposition}
\newtheorem{question}[theorem]{Question}
\theoremstyle{definition}
\newtheorem{remark}[theorem]{Remark}
\newtheorem{definition}[theorem]{Definition}
\newtheorem{construction}[theorem]{Construction}

\newtheorem{example}[theorem]{Example}

\def\cF{{\mathcal F}}

\def\cS{{\mathcal S}}

\def\x{{{\bf x}}}

\def\bQ{{\mathbb Q}}
\def\bR{{\mathbb R}}

\def\bZ{{\mathbb Z}}

\def\del{{\partial}}

\def\Z{\mathbb{Z}}
\def\Q{\mathbb{Q}}

\def\T{\mathbb{T}}

\def\SS{\mathfrak{S}}

\def\CF {\widehat{\operatorname{CF}}}
\def\HF {\widehat{\operatorname{HF}}}
\def\CFK {\widehat{\operatorname{CFK}}}
\def\HFK {\widehat{\operatorname{HFK}}}

\def\spincs {\mathfrak{s}}

\newcommand{\RP}{\mathbb{R} \mathrm{P}}

\newcommand{\abs}[1] {\left\lvert #1 \right\rvert}

\newcommand{\gen}[1] {\langle #1 \rangle}
\def\Th{^{\text{th}}}
\def\minus{\smallsetminus}
\def\co{\colon\thinspace}

\def\conn{\mathbin{\#}}

  \DeclareMathOperator{\rank}{rank}
\DeclareMathOperator{\coker}{coker}  
\DeclareMathOperator{\Sym}{Sym} \DeclareMathOperator{\Spin}{Spin}
 
\DeclareMathOperator{\PD}{PD} 
\DeclareMathOperator{\sgn}{sgn} \DeclareMathOperator{\sign}{sign}
\DeclareMathOperator{\per}{per}

\definecolor{purple}{rgb}{1,0,1}
\definecolor{darkblue}{rgb}{0,0,0.5}
\definecolor{darkred}{rgb}{0.5,0,0}
\definecolor{darkgreen}{rgb}{0,0.5,0}

\begin{document}

\title[Strong Heegaard diagrams and strong L-spaces]%
{Strong Heegaard diagrams and strong L-spaces}

\author[Joshua Evan Greene]{Joshua Evan Greene}
\address{Department of Mathematics, Boston College \\ Chestnut Hill, MA 02467}
\email{joshua.greene@bc.edu}

\author[Adam Simon Levine]{Adam Simon Levine}
\address{Department of Mathematics, Princeton University \\ Princeton, NJ 08540}
\email{asl2@math.princeton.edu}

\thanks{JEG was supported by NSF grant DMS-1207812 and an Alfred P. Sloan Research Fellowship.  ASL was supported by NSF grants DMS-1004622 and DMS-1405378.}

\begin{abstract}
We study a class of 3-manifolds called strong L-spaces, which by definition admit a certain type of Heegaard diagram that is particularly simple from the perspective of Heegaard Floer homology.  We provide evidence for the possibility that every strong L-space is the branched double cover of an alternating link in the three-sphere.  For example, we establish this fact for a strong L-space admitting a strong Heegaard diagram of genus two via an explicit classification.  We also show that there exist finitely many strong L-spaces with bounded order of first homology; for instance, through order eight, they are connected sums of lens spaces.  The methods are topological and graph theoretic.  We discuss many  related results and questions.
\end{abstract}

\maketitle

%%%%%%
%%%%%%
%%%%%%
%%%%%%
%%%%%%
%%%%%%

\section{Introduction} \label{sec: introduction}

The purpose of this paper is to study a family of $3$-manifolds called {\em strong L-spaces}.  These manifolds are defined by a combinatorial condition on Heegaard diagrams, and they arise naturally in the context of Heegaard Floer homology.

In its simplest form, the Heegaard Floer homology \cite{OSz3Manifold} of a closed, oriented $3$-manifold $Y$ is a finitely generated abelian group $\HF(Y)$, defined as follows. We present $Y$ by means of a Heegaard diagram $H$, consisting of a closed, oriented surface $S$ of genus $g$ and two disjoint unions of embedded circles in $S$, $\alpha = \alpha_1 \cup \dots \cup \alpha_g$ and $\beta = \beta_1 \cup \dots \cup \beta_g$, each of which spans a $g$-dimensional subspace of $H_1(S;\Z)$ and which intersect each other transversally. To such a diagram, we associate a chain complex $\CF(H)$, which is freely generated by the set $\SS(H)$ of unordered $g$-tuples of points in $S$ with one point on each $\alpha$ circle and one point on each $\beta$ circle.
%We may also view $\SS(H)$ as the set of intersection points between two $g$-dimensional tori in a $2g$-dimensional manifold, the $g\Th$ symmetric product of $S$.
By adapting the machinery of Lagrangian Floer homology, Ozsv\'ath and Szab\'o define a differential $\partial$ on $\CF(H)$ that also depends on some additional choices of analytic data.  They prove that the homology $H_*(\CF(H),\partial)$ depends only on the $3$-manifold $Y$ and not on the specific choice of Heegaard diagram or analytic data.  This homology group is denoted $\HF(Y)$.

%$g$-tuples of pairwise-disjoint curves $\alpha = (\alpha_1, \dots, \alpha_g)$ and $\beta = (\beta_1, \dots, \beta_g)$, each of which spans a $g$-dimensional subspace of $H_1(S;\Z)$, along with a basepoint $z \in S \minus (\bigcup \alpha_i \cup \bigcup \beta_i)$. (The $\alpha$ and $\beta$ curves are required to intersect transversally.)

Define the \emph{determinant} $\det(Y)$ of a $3$-manifold $Y$ to be the order of $H_1(Y;\Z)$ if this group is finite (i.e. when $Y$ is a rational homology sphere) and $0$ otherwise.
%Ozsv\'ath and Szab\'o showed that
With respect to a natural $\Z/2$--grading, the Euler characteristic of $\CF(Y)$ is equal to $\det(Y)$ \cite{OSzProperties}. As a result, for any Heegaard diagram $H$ presenting $Y$, we have
\begin{equation}
\label{eq: CF-HF-det}
\abs{\SS(H)} = \rank \CF(H) \ge \rank \HF(Y) \ge \det(Y).
\end{equation}
If $\det(Y) \ne 0$ and $\HF(Y)$ is free abelian of rank $\det(Y)$, $Y$ is called an \emph{L-space}.  In view of \eqref{eq: CF-HF-det}, such manifolds can be seen as having the simplest possible Heegaard Floer homology.\footnote{Some authors define $Y$ to be an L-space under the weaker condition that $\rank \HF(Y) = \det(Y)$, which can be easier to verify. We refer to such a $Y$ as an \emph{L-space mod torsion}. In fact, there is no known example of a rational homology sphere $Y$ for which $\HF(Y)$ contains torsion; the two definitions may be equivalent.} Examples of L-spaces include $S^3$, lens spaces (whence the name), all manifolds with finite fundamental group, and branched double covers of non-split alternating (or, more generally, quasi-alternating) knots and links in $S^3$ \cite{OSzDouble}. The classification of L-spaces is one of the major outstanding questions in Heegaard Floer theory. For instance, it is conjectured that a rational homology sphere $Y$ is an L-space if and only if $\pi_1(Y)$ is not left-orderable.  This conjecture is known to hold (at least mod torsion) for many classes of manifolds, including all geometric, non-hyperbolic manifolds \cite{BoyerGordonWatson}.

The minimum size of $\SS(H)$, ranging over all Heegaard diagrams $H$ for a rational homology sphere $Y$, may be viewed as a measure of the topological complexity of $Y$.  This quantity is called the \emph{simultaneous trajectory number} of $Y$ and is denoted $M(Y)$ \cite[Section 1.2]{OSzProperties}.  Both $\det(Y)$ and $\rank(\HF(Y))$ provide lower bounds on $M(Y)$ by \eqref{eq: CF-HF-det}. The second author and Lewallen \cite{LevineLewallen} introduced the following definition:
\begin{definition}
\label{d: SLS}
A closed, oriented $3$-manifold $Y$ is called a \emph{strong L-space} if $M(Y) = \det(Y)$.  A Heegaard diagram $H$ for $Y$ for which $\abs{\SS(H)} = \det(Y)$ is called a \emph{strong Heegaard diagram}.
\end{definition}
\noindent
By \eqref{eq: CF-HF-det}, a strong L-space is an L-space. In view of the conjecture mentioned above, the second author and Lewallen proved that the fundamental group of a strong L-space is not left-orderable \cite{LevineLewallen}. As a simple direct consequence, a strong L-space cannot admit an $\bR$-covered taut foliation.  In fact, Ozsv\'ath and Szab\'o established the much deeper result that an L-space does not admit any co-orientable taut foliation \cite{KazezRoberts,OSzGenus}.

Our motivating problem is to describe the homeomorphism types of strong L-spaces and strong Heegaard diagrams. As we shall see, the condition of being a strong L-space is quite restrictive, more so than being an L-space. For instance, the only strong L-space with determinant $1$ is $S^3$, so the Poincar\'e homology sphere is an L-space but not a strong L-space \cite{OSzProperties,LevineLewallen}.

The standard genus-$1$ Heegaard diagram for a lens space $L(p,q)$ (consisting of a single $\alpha$ curve and a single $\beta$ curve on a torus, intersecting $p$ times) is clearly a strong diagram, so lens spaces are strong L-spaces.  A broader source of examples derives from work of the first author, who showed that the double cover of $S^3$ branched along a non-split alternating link is a strong L-space \cite[Corollary 3.5]{GreeneSpanning}.  Note that these spaces subsume lens spaces, which are branched double covers of two-bridge links, although the strong diagrams for these spaces in \cite{GreeneSpanning} do not have genus $1$. Although we can generate many families of strong Heegaard diagrams, we were unable to produce any new examples of strong L-spaces besides the ones just mentioned.  Indeed, our main results support an affirmative answer to the following question:

\begin{question}
\label{q: alternating}
Is every strong L-space the branched double cover of an alternating link in $S^3$?
\end{question}

As a first approach to Question \ref{q: alternating}, recall that the determinant of a link $L \subset S^3$ is the absolute value of its (single-variable) Alexander polynomial evaluated at $-1$: $\det(L) = \abs{\Delta_{L}(-1)}$. Let $\Sigma(L)$ denote the double cover of $S^3$ branched over $L$.  Then $\det(L) = \det(\Sigma(L))$, in part justifying our terminology.  A classical theorem of Bankwitz and Crowell asserts that there exist finitely many alternating links of bounded determinant \cite{BankwitzAlternating,CrowellNonalternating}.  Therefore, an affirmative answer to Question \ref{q: alternating} would imply the same about strong L-spaces.  We deduce this fact directly as a corollary of the following result, which we prove using topological and graph-theoretic methods:

\begin{theorem}
\label{t: finite}
There exist finitely many rational homology spheres with bounded simultaneous trajectory number.
\end{theorem}

\begin{corollary}
\label{c: finite}
There exist finitely many strong L-spaces with bounded determinant. \qed
\end{corollary}

\noindent

By contrast, there exist infinitely many irreducible L-spaces with the same determinant.  Reducible examples are easy to exhibit, since Heegaard Floer homology satisfies a K\"unneth principle for connected sums.  Thus, for example, the connected sum of any L-space with arbitrarily many copies of the Poincar\'e sphere is an L-space with the same determinant.  For irreducible examples, the Seifert fibered spaces of type $(2,2,n)$ all have determinant $4$ and finite fundamental group, so they are L-spaces. (It is unknown whether there exist infinitely many irreducible L-spaces with determinant less than $4$.) Additionally, the first author and Watson \cite{GreeneWatsonQA} gave an infinite family of hyperbolic manifolds with determinant $25$ that are L-spaces mod torsion.

Using similar techniques to those in the proof of Theorem \ref{t: finite}, we prove:

\begin{theorem} \label{t: le8}
If $Y$ is a strong L-space with $\det(Y) \le 8$, then $Y$ is the branched double cover of an alternating link.  Specifically, $Y$ is a connected sum of lens spaces.
\end{theorem}

\begin{figure}%[htb!]
\centering
\includegraphics[width=2in]{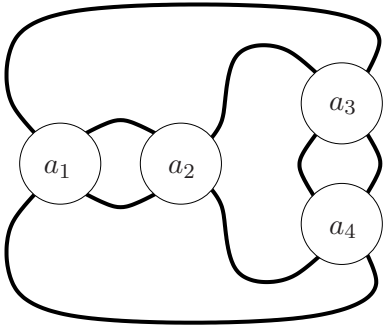}
\put(-129,58){$a_1$}
\put(-82,58){$a_2$}
\put(-21,82){$a_3$}
\put(-21,35){$a_4$}
\caption{Diagram of the link $L(a_1,a_2,a_3,a_4)$, $a_1,a_2,a_3,a_4 \in \Q \cup \{1/0\}$.}  \label{f:branchset}
\end{figure}

In another direction towards Question \ref{q: alternating}, we describe all strong L-spaces that admit strong Heegaard diagrams of genus 2. For $a_1, a_2, a_3, a_4 \in \Q \cup \{1/0\}$, let $L(a_1, a_2, a_3, a_4)$ denote the link presented by the the diagram in Figure \ref{f:branchset}, where the four balls are filled in with the rational tangles specified by $a_1, a_2, a_3, a_4$ (see Section \ref{ss:tangles}). Note that the diagram is alternating iff all $a_i$ have the same sign, where $0/1$ and $1/0$ are considered to have both signs. Setting $a_1 = a_2 = a_3 = a_4 = \pm1$ results in the minimal diagram $D$ of the figure eight knot, so the diagram can be regarded as substituting arbitrary rational tangles for the crossings in $D$. The branched double cover $\Sigma(L(a_1,a_2,a_3,a_4))$ is either
\begin{inparaenum}
%\item a connected sum of zero, one, or two lens spaces or copies of $S^1 \times S^2$,
\item a connected sum of one or two genus-1 manifolds,
\item a small Seifert fibered space, or
\item a graph manifold whose JSJ decomposition consists of two Seifert fibered spaces over $D^2$ with two exceptional fibers.
\end{inparaenum}

\begin{theorem} \label{t: genus2}
Let $Y$ be a strong L-space that admits a strong Heegaard diagram of genus $2$. Then for some $a_1, \dots, a_4 \in \Q \cup \{1/0\}$ of the same sign, $Y \cong \Sigma(L(a_1,a_2,a_3,a_4))$. In particular, $Y$ is the branched double cover of an alternating link in $S^3$.
\end{theorem}

\noindent
Usui \cite{UsuiSmoothing1, UsuiSmoothing2} has obtained similar results.

%As further evidence concerning Question \ref{q: alternating}, Nathan Dunfield used a computer to search for strong L-spaces among the $11,031$ small hyperbolic $3$-manifolds in the Hodgson--Weeks census, and identified $316$ such manifolds, all of which can be realized as branched double covers of alternating links. (See Example \ref{ex: dunfield} below.)
%
A key ingredient for proving Theorems \ref{t: finite} and \ref{t: genus2} is Proposition \ref{p: 1-extendible}, which implies that every strong L-space admits a Heegaard diagram $H$ that is both strong and \emph{1-extendible}. In such a diagram, the signs with which the $\alpha$ and $\beta$ curves may intersect are highly constrained: all points of intersection between any two curves have the same sign, and the associated intersection matrix is  a P\'olya matrix (see Section \ref{sec: preliminaries}). Theorems \ref{t: finite} and \ref{t: le8} then follow from topological and graph theoretic arguments, and their proofs appear in Section \ref{sec: finiteness}. To prove Theorem \ref{t: genus2}, we show in Section \ref{sec: genus2} that every strong, 1-extendible Heegaard diagram of genus 2 has a standard form, which coincides precisely with a particular class of Heegaard diagrams for $\Sigma(L(a_1,a_2,a_3,a_4))$ for $a_1, a_2, a_3, a_4$ of the same sign.

In Section \ref{sec: simpleknots}, we prove some results concerning Floer simple knots in strong L-spaces that admit genus-2 strong diagrams.  The existence of such knots has applications to Dehn surgery and minimal genus problems. In Section \ref{sec: waves}, we discuss the connections between strong L-spaces and other notions pertaining to Heegaard splittings.  We close in Section \ref{sec: questions} with several questions motivated by the present work.

%A direct application of graph-theoretic results implies the following result. \smargin{I think this may be too technical a result for the intro.}
%
%\begin{proposition}
%Every strong 1-extendible Heegaard diagram contains an $\alpha$ curve that meets at most three $\beta$ curves, and vice versa.
%\end{proposition}
%
%\noindent
%Thus, every strong 1-extendible Heegaard diagram of genus 3 or more is weakly reducible, a slight disadvantage to our terminology.

\subsection*{Acknowledgments}
%We would like to extend our foremost thanks to John Luecke.  We carried out a lot of the work on this paper with him, but he ultimately declined to sign on as our co-author.  We warmly thank Cameron Gordon and Sam Lewallen as well for enjoyable, stimulating discussions.
We extend our foremost thanks to John Luecke, who had a great influence on this work. We also warmly thank Cameron Gordon, Sam Lewallen, and Nathan Dunfield for enjoyable, stimulating discussions.

\section{Preliminaries} \label{sec: preliminaries}

For an oriented, properly embedded curve $\alpha$ in a surface (possibly with boundary), let $-\alpha$ denote the same curve with the opposite orientation, and for any integer $n$, let $n \alpha$ denote a multi-curve obtained by taking $|n|$ parallel copies of $\sgn(n) \alpha$.  For oriented multi-curves $\alpha$ and $\beta$ that meet transversally, let $\alpha + \beta$ denote an oriented multi-curve obtained from $\alpha \cup \beta$ by forming the oriented resolution at every point of $\alpha \cap \beta$.  Note that $n \alpha$ and $\alpha + \beta$ specify unique multi-curves up to isotopy.  For oriented multi-curves $\alpha, \beta$ in an oriented surface $S$ that meet transversally, write $\alpha \cdot \beta$ for their algebraic intersection number and $\abs{\alpha \cap \beta}$ for their geometric intersection number. The multi-curves $\alpha$ and $\beta$ intersect \emph{coherently} if all intersection points have the same sign, i.e., if $\abs{\alpha \cdot \beta} = \abs{\alpha \cap \beta}$.

A Heegaard diagram $H$ consists of a surface $S$, a basepoint $z$, two collections of attaching curves $\alpha = \alpha_1 \cup \cdots \cup \alpha_g$ and $\beta = \beta_1 \cup \cdots \cup \beta_g$ specifying a pair of handlebodies, and a choice $o$ of orientation on $S$ and each $\alpha$ and $\beta$ curve:
\[
(S,z,\alpha,\beta,o).
\]
Henceforth, we generally suppress reference to the basepoint $z$ and orientation $o$, and simply refer to the diagram as $(S,\alpha,\beta)$.

Let $\Sym^g(S)$ denote the $g\Th$ symmetric power of $S$, the quotient of $S^{\times g}$ by the action of the symmetric group on $g$ letters.  It is a $2g$-dimensional manifold. Let $\T_\alpha$ (resp.~$\T_\beta$) be the image in $\Sym^g(S)$ of $\alpha_1 \times \cdots \times \alpha_g$ (resp.~$\T_\beta$); this is an embedded $g$-dimensional torus. The tori $\T_\alpha$ and $\T_\beta$ intersect transversally in a finite number of points; let $\SS(H) = \T_\alpha \cap \T_\beta$. A point of $\SS(H)$ is a tuple $\x = (x_1, \dots, x_g)$, where $x_i \in \alpha_i \cap \beta_{\sigma_{\x}(i)}$, for some permutation $\sigma_\x$. Elements of $\SS(H)$ are \emph{generators} of the group $\CF(H)$ discussed above.

%\smargin{Or we could avoid using the term generator.}  We use it later and I think it is fine. JG

%\smargin{Notation for symmetric group? $S_g$? $\mathcal{S}_g$?}  It seems we get away without having to notate it. JG

%Recall that we have fixed orientations on $S$ and on all the $\alpha$ and $\beta$ curves.
For each $x \in \alpha_i \cap \beta_j$, let $\eta(x) \in \{\pm 1\}$ denote the local sign of intersection of $\alpha_i$ with $\beta_j$ at $x$. For $\x = (x_1, \dots, x_g) \in \SS(H)$, the local sign of intersection of $\T_\alpha$ and $\T_\beta$ is given by
\[
\eta(\x) = \sign(\sigma_\x) \prod_{i=1}^g \eta(x_i)
\]
where $\sign(\sigma_\x) \in \{\pm1\}$ is the signature of the permutation $\sigma_\x$.
% ($+1$ for an even permutation, $-1$ for an odd permutation).
Note that changing the orientation of $S$ or a single $\alpha$ or $\beta$ curve negates $\eta(\x)$ for all $\x \in \SS(H)$.

The {\em intersection matrix} $M(H)$ is the $g \times g$ matrix of integers whose $(i,j)$ entry is
\[
m_{i,j} = \alpha_i \cdot \beta_j = \sum_{x \in \alpha_i \cap \beta_j} \eta(x).
\]
Since $M(H)$ is a presentation matrix for $H_1(Y)$, we have $\det(Y) = \abs{\det(M(H))}$, which explains our choice of terminology. The permutation expansion of the determinant gives:
\begin{align*}
\det (M(H)) &= \sum_\sigma \sign(\sigma) m_{1,\sigma(1)} \cdots m_{g, \sigma(g)} \\
&= \sum_\sigma \sign(\sigma) \sum_{\substack{ \{(x_1, \dots, x_g) \mid \\ x_i \in \alpha_i \cap \beta_{\sigma(i)}\}}} \prod_{i=1}^g \eta(x_i) \\
&= \sum_{\x \in \SS(H)} \eta(\x).
\end{align*}
In particular, we see that
\[
\abs{\SS(H)} \ge \det(Y),
\]
as noted in the Introduction. The following two lemmas are immediate:
%Furthermore, we immediately have:

\begin{lemma}
\label{l: samesign}
$H$ is a strong Heegaard diagram if and only if all generators in $\SS(H)$ have the same sign. \qed
\end{lemma}

%As an immediate consequence of Lemma \ref{l: samesign}, we have:

\begin{lemma} \label{l: coherent}
If $H$ is a strong Heegaard diagram, and some generator in $\SS(H)$ includes a point of $\alpha_i \cap \beta_j$, then $\alpha_i$ and $\beta_j$ intersect coherently. \qed
\end{lemma}

We call a Heegaard diagram $H$ {\em coherent} if all pairs of $\alpha$ and $\beta$ curves intersect coherently. As we shall see, the fact that (many of) the curves in a strong Heegaard diagram intersect coherently will be vital to our classification results. However, a strong diagram need not be coherent. For instance, if $H$ is a genus-2 Heegaard diagram in which the intersections $\alpha_1 \cap \beta_1$ and $\alpha_2 \cap \beta_2$ are coherent and non-empty, while $\alpha_1 \cap \beta_2 = \emptyset$, then $H$ is strong irrespective of the signs of points in $\alpha_2 \cap \beta_1$, since these points are not included in elements of $\SS(H)$. Such a diagram can be obtained by taking the connected sum of two standard Heegaard diagrams for lens spaces and isotoping the $\alpha$ curve of one summand so that it intersects the $\beta$ curve of the other summand. However, Proposition \ref{p: 1-extendible} below will enable us to restrict our attention to coherent Heegaard diagrams.

Given a Heegaard diagram $H$, the \emph{intersection graph} $G(H)$ is the bipartite graph with vertex set $A \sqcup B$,  where $A = \{a_1,\dots,a_g\}$ and $B = \{b_1,\dots,b_g\}$, and for which the set of edges joining $a_i$ and $b_j$ is $\{e_x \mid x \in \alpha_i \cap \beta_j\}$. Note that $\SS(H)$ is in natural one-to-one correspondence with the set of perfect matchings of $G(H)$.
%We orient $G(H)$ by the rule that $e_x$ is directed from $A$ to $B$ if $\eta(x) = 1$ and from $B$ to $A$ if $\eta(x) = -1$.
The {\em degree} of a vertex $v$ in a graph $G$ is the number of edges in $G$ with an endpoint at $v$.  Write $\delta(G)$ for the minimum vertex degree  in $G$, and write $\delta(H) = \delta(G(H))$.  A graph $G$ is {\em $k$-extendible} if any $k$-tuple of disjoint edges extends to a perfect matching of $G$.  A Heegaard diagram $H$ is {\em $k$-extendible} if its intersection graph $G(H)$ is. In particular, $H$ is $1$-extendible if every point $x \in \alpha \cap \beta$ is contained in a generator $\x \in \SS(H)$. As an immediate consequence of the previous two lemmas, we have:

\begin{lemma} \label{l: all coherent}
A strong, $1$-extendible Heegaard diagram is coherent. \qed
\end{lemma}

For a strong diagram $H$, the combinatorics of $M(H)$ and $G(H)$ are expressed by the closely related concepts of P\'olya matrices and Pfaffian orientations, respectively.  A {\em P\'olya matrix} is a $g \times g$ square matrix $M = (m_{i,j})$ for which all non-zero terms in the expansion
\[
\det(M) = \sum_\sigma \sign(\sigma) m_{1,\sigma(1)} \cdots m_{g, \sigma(g)}
\]
come with the same sign.  Equivalently, if we write $\abs{M}$ for the matrix $(\abs{m_{i,j}})$ and
\[
\per(N) = \sum_\sigma n_{1,\sigma(1)} \cdots n_{g, \sigma(g)}
\]
for the {\em permanent} of a $g \times g$ square matrix $N = (n_{i,j})$, then $M$ is a P\'olya matrix if
\[
\abs{\det(M)} = \per(\abs{M}).
\]
We refer to \cite{vy:pfaffian} for the definition of a Pfaffian orientation.
%, as we shall not require it in what follows.
%A cycle $C$ in a bipartite graph $G$ is {\em central} if there is a perfect matching in the subgraph of $G$ that results after deleting all vertices and edges incident with $C$.  If $G$ is oriented, then $C$ is {\em oddly oriented} if an odd number of edges are directed in $G$ around $C$ in the direction of each orientation of $C$.  An orientation on $G$ is {\em Pfaffian} if every central cycle is oddly oriented.

\begin{proposition}
\label{p: polya}
The following are equivalent conditions on a coherent Heegaard diagram $H$:
\begin{inparaenum}
\item
$H$ is strong,
\item
$M(H)$ is a P\'olya matrix, and
\item
$G(H)$ has a Pfaffian orientation.
\end{inparaenum}
\end{proposition}

\begin{proof}
The equivalence of (1) and (2) is the effective content of Lemma \ref{l: samesign}.  The equivalence of (2) and (3) follows from \cite{vy:pfaffian}.
\end{proof}

The sign pattern of the entries of a P\'olya matrix is highly constrained.  This fact plays a key role in the proof that the fundamental group of a strong L-space is not left-orderable \cite{LevineLewallen}. P\'olya matrices obey a deep structure theorem due independently to McCuaig \cite{mccuaig:polya} and Robertson, Seymour, and Thomas  \cite{rst:polya}.  We apply a result of their work in Section \ref{subsec: reducibility}.

\section{Extendibility} \label{sec: extendibility}

The purpose of this section is to show that every rational homology sphere admits a 1-extendible Heegaard diagram attaining its simultaneous trajectory number.  In particular, every strong L-space admits a 1-extendible, strong Heegaard diagram.  This result is very useful in the subsequent sections. The technical statement is as follows:

\begin{proposition}
\label{p: 1-extendible}
Let $H$ be a doubly-pointed Heegaard diagram for a rational homology sphere.  Then there exists a sequence of handleslides and isotopies in the complement of the basepoints that transforms $H$ into a 1-extendible Heegaard diagram $H'$ such that $\abs{\SS(H)} = \abs{\SS(H')}$. In particular, if $H$ is strong, then so is $H'$.
\end{proposition}

\begin{corollary}
\label{c: 1-extendible}
A rational homology sphere $Y$ admits a 1-extendible Heegaard diagram $H$ for which $\abs{\SS(H)} = M(Y)$. \qed
\end{corollary}

To prove Proposition \ref{p: 1-extendible}, we begin by establishing a simple criterion to recognize that a given Heegaard diagram can be converted into a reducible one by a sequence of isotopies and handleslides.  In order to state it in a sharp form, we require a little notation and background concerning Heegaard diagrams and presentations of the fundamental group.
%A Heegaard diagram for a 3-manifold $Y$ yields natural presentations for its fundamental group $\pi_1(Y)$.  We review this now in order to introduce the {\em cyclic word} associated to a curve.

Given any free group $F^z = \gen{ z_1,\dots,z_g }$ and a value $1 \leq k \leq g$, we obtain groups
\[
F^z_{\leq k} = \gen{ z_1, \dots, z_k }, \quad  F^z_k = \gen{ z_k }, \quad F^z_{\geq k} = \gen{ z_k, \dots, z_g },
\]
and projections from $F^z$ to each. For any word $w \in F^z$, let
\[
w_{\leq k} = w_{< k+1}, \quad w_k,  \quad \text{and} \quad w_{\geq k} = w_{> k-1},
\]
denote the respective images of $w$ under these projections; each of them is a subword of $w$. In a similar spirit, for a collection of curves $\gamma_1,\dots,\gamma_g$, we define
\[
\gamma_{\le k} = \gamma_{<k+1} =  \gamma_1 \cup \cdots \cup \gamma_k \quad \text{and} \quad \gamma_{ > k} = \gamma_{k+1} \cup \cdots \cup \gamma_g.
\]

Let $H = (S,\alpha,\beta)$ denote a Heegaard diagram that presents a 3-manifold $Y$.  Form the free group $F^x = \gen{x_1, \dots, x_g}$.  Choose a point $p_i \in \alpha_i$ disjoint from $\beta$ and traverse a full loop around $\alpha_i$ starting at $p_i$.  For each intersection point between $\alpha_i$ and $\beta_j$ with sign $\epsilon$, record the term $x_j^\epsilon$.  The product of these terms, in order from left to right, yields a word $w_\beta(\alpha_i) \in F_x$.  Observe that a different choice of $p_i$  will result in the same word up to cyclic equivalence.  As a result, we regard $w_\beta(\alpha_i)$ and its subwords up to cyclic equivalence. The group $\pi_1(Y)$ admits the presentation $F^x / \gen{ w_\beta(\alpha_1), \dots, w_\beta(\alpha_g)}$.  In an analogous manner, we define the free group $F^y = \gen{ y_1, \dots, y_g }$, associate a cyclic word $w_\alpha(\beta_j)$ in $y_1, \dots, y_g$ to each curve $\beta_j$, and obtain the presentation $F^y / \gen{ w_\alpha(\beta_1), \dots, w_\alpha(\beta_g) }$ for $\pi_1(Y)$.

\begin{lemma}
\label{l: connected sum}
Suppose that $H = (S,\alpha,\beta,z_1,z_2)$ is a doubly-pointed genus-$g$ Heegaard diagram for a rational homology sphere $Y$, and for some $0 < k < g$,
\[
\alpha_{\le k} \cap \beta_{> k} = \emptyset.
\]
%\[
%(\alpha_1 \cup \cdots \cup \alpha_k) \cap (\beta_{k+1} \cup \cdots \cup \beta_g) = \emptyset.
%\]
Then it is possible to perform handleslides and isotopies in the complement of $z_1$ and $z_2$ to produce a Heegaard diagram $H' = (S,\alpha',\beta',z_1,z_2)$ such that $\abs{\SS(H)} = \abs{\SS(H')}$,
\[
w_{\beta'}(\alpha_i') = w_\beta(\alpha_i)_{\leq k}, \quad w_{\alpha'}(\beta_j') = w_\alpha(\beta_j)_{\leq k}, \quad \forall \, i,j \leq k,
\]
and
\[
w_{\beta'}(\alpha_i') = w_\beta(\alpha_i)_{> k}, \quad w_{\alpha'}(\beta_j') = w_\alpha(\beta_j)_{> k}, \quad \forall \, i,j > k.
\]
In particular, $\alpha_{\le k}' \cap \beta_{> k}' = \alpha_{>k}' \cap \beta_{\le k} = \emptyset$, while there is an identification $\alpha_i' \cap \beta_j' = \alpha_i \cap \beta_j$ preserving local signs of intersection when $i,j \le k$ or $i,j > k$. Thus, $H'$ is reducible with summands $H_1 = (S_{\leq k},\alpha_{\leq k}', \beta_{\leq k}')$ and $H_2 = (S_{> k},\alpha_{> k}', \beta_{> k}')$.  If $H$ is strong, then so are $H_1$ and $H_2$.
\end{lemma}

\begin{remark}
Some restriction on $Y$ or $H$ is necessary in order to guarantee the conclusion of Lemma \ref{l: connected sum}.  For example, take $Y = S^1 \times S^2$ (so $b_1(Y) > 0$), consider its standard genus-$1$ Heegaard diagram with no intersection points, and stabilize the diagram once.  Label the curves in the resulting diagram so that $\alpha_1 \cap \beta_2 = \emptyset$.  The conclusion of Lemma \ref{l: connected sum} in this case (with $k=1$) would produce a genus-$2$ Heegaard diagram in which each pair of curves is disjoint, but such a diagram must present $\#^2 (S^1 \times S^2) \ne Y$, a contradiction.
\end{remark}

\begin{proof}[Proof of Lemma \ref{l: connected sum}]
Recall that $H_1(Y;\Z) \cong H_1(S;\Z) / \gen{[\alpha_1], \dots, [\alpha_g], [\beta_1], \dots, [\beta_g]}.$  Since $Y$ is a rational homology sphere, it follows that the classes $[\alpha_1], \dots, [\alpha_g], [\beta_1], \dots, [\beta_g]$ are linearly independent in $H_1(S;\Z)$.  In particular, $\alpha_1,\dots, \alpha_k, \beta_{k+1}, \dots, \beta_g$ are $g$ disjoint curves representing linearly independent classes in $H_1(S;\Z)$.  It follows that $S \minus (\alpha_{\le k} \cup \beta_{> k})$ is a $2g$-punctured sphere.  As a result, there exists an unoriented curve $\gamma \subset S$, unique up to isotopy, with a regular neighborhood $N(\gamma)$ that separates $S$ into subsurfaces $\overline{S}_{\leq k}, \overline{S}_{> k}$ such that $\{ z_1 \} \cup \alpha_{\le k} \subset \overline{S}_{\leq k}$ and $\{ z_2 \} \cup \beta_{> k} \subset \overline{S}_{> k}$. See Figure \ref{fig: extendible}.

\begin{figure}
\labellist
 \pinlabel {{\color{red} $\alpha_1$}}  at 112 260
 \pinlabel {{\color{red} $\alpha_2$}}  at 195 274
 \pinlabel {{\color{red} $\alpha_3$}}  at 281 286
 \pinlabel {{\color{blue} $\beta_4$}}  at 128 98
 \pinlabel {{\color{blue} $\beta_5$}}  at 250 103
\pinlabel{$\overline{S}_{\le 3}$} [bl] at 309 301
\pinlabel{$\overline{S}_{> 3}$} [tl] at 309 56
\pinlabel{$z_1$} [l] at 183 344
\pinlabel{$z_2$} [l] at 183 13
\pinlabel{$\gamma$} [tl] at 294 159
\endlabellist
\includegraphics[scale=0.5]{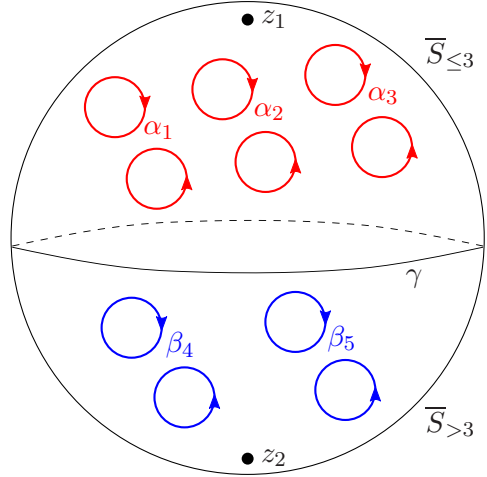}
\caption{A cut-open Heegaard diagram as in the proof of Lemma \ref{l: connected sum}, with $k=3$ and $g=5$.} \label{fig: extendible}
\end{figure}

%$S - (\alpha_1 \cup \cdots \cup \alpha_k \cup \beta_{k+1} \cup \cdots \cup \beta_g)$

The space $N(\gamma) \cup \overline{S}_{> k} \minus \beta_{>k}$ is a $2(g-k)$-punctured disk, and $N(\gamma) \cap \overline{S}_{\le k}$ is a collar neighborhood of its boundary.  We radially isotope $\beta_{\le k} \cap (N(\gamma) \cup \overline{S}_{> k})$ away from $z_2$ and into $N(\gamma) \cap \overline{S}_{\leq k}$, permitting arcs to pass over punctures in the process. Passing an arc over a puncture corresponds to a handleslide in $S$, so we may interpret this process as a sequence of isotopies and handleslides of $\beta_{\le k}$ over $\beta_{> k}$ in $S$ which is the identity outside of $N(\gamma) \cup \overline{S}_{> k} \minus \{z_2\}$.  The resulting collection of curves $\beta_{\le k}'$ supported in $\overline{S}_{\leq k}$ has the property that $w_\alpha(\beta_j')_{\leq k} = w_\alpha(\beta_j)_{\leq k}$ for all $j \leq k$, since $\alpha_{\le k} \cap \beta_{> k} = \emptyset$ and the handleslides do not change the intersections of $\beta_{\le k}$ with $\alpha_{\le k}$. Similarly, we isotope and handleslide the curves $\alpha_{> k}$ over $\alpha_{\le k}$ in $N(\gamma) \cup \overline{S}_{\le k} \minus \{z_1\}$ to get curves $\alpha_{> k}'$ supported in $\overline{S}_{> k}$ and such that $w_{\beta'}(\alpha_i')_{> k} = w_{\beta'}(\alpha_i)_{> k}$ for all $i$.  Lastly, we define $\alpha_i' = \alpha_i$ for $i =1, \dots, k$ and $\beta_j' = \beta_j$ for $j=k+1, \dots, g$.

The preceding remarks and the fact that $\beta_{\le k}' \cap \alpha_{> k}' = \emptyset$ yield the stated conclusions about $w(\alpha_i')$ and $w(\beta_j')$.  Capping off $\overline{S}_{\leq k}$ and $\overline{S}_{> k}$ along $\gamma$ results in the surfaces $S_{\leq k}$ and $S_{> k}$ required by the Lemma.  Finally, note that there exists a natural bijection between $\SS(H)$ and $\SS(H')$, since a generator in $\SS(H)$ cannot use any element of $\alpha_{>k} \cap \beta_{\le k}$. In particular, if $H$ is strong, then so is $H'$.
\end{proof}

\begin{proof}[Proof of Proposition \ref{p: 1-extendible}]
We proceed by induction on $g(H)$.  For $g(H)=1$ the assertion is trivial, so we proceed to the induction step.  By a theorem of Hetyei, a bipartite graph $G$ with bipartition $(A,B)$ is 1-extendible iff for every subset $T \subset A$, the set of neighbors of $T$ in $B$ has cardinality at least $\max\{|B|,|A|+1\}$ \cite{hetyei}, \cite[Theorem 4.1.1]{LovaszPlummerBook}.  Since $G(H)$ contains a perfect matching, it is 0-extendible.  Assume that it is not 1-extendible.  Then there exists some proper, non-empty subset $T \subset A$ such that $T$ has exactly $|T|$ neighbors in $B$.  In other words, there exist some $k$ $\alpha$ circles that are disjoint from some $(g-k)$ $\beta$ circles, where $1 \leq k < g$.  Thus, Lemma \ref{l: connected sum} implies that there exists a sequence of handleslides and isotopies which transform $H$ into a reducible diagram $H'$, and which is strong provided that $H$ is.  By induction on the genera of the summands of $H'$, the statement of the Proposition follows.
\end{proof}

A simple induction using Lemma \ref{l: connected sum} establishes the following corollary:

\begin{corollary} \label{c: lens spaces}
If a strong diagram has an upper triangular intersection matrix, then it presents a connected sum of lens spaces. \qed
\end{corollary}

%\section{ A truly general result \\ pertaining to the $\langle\langle${\bf \em alchemy}$\rangle\rangle$ of \\ {\small the so-called} \underline{\Large simultaneous trajectory number} \\ $\dots$ {\em along with a variation on} $\dots$ \\ the {\bf re{\em mark}able} case of a \\ $\star$ \textcolor{red}{strong \textcolor{blue}{L}-space} $\star$ \\ {\footnotesize of determinant at most eight}  %The finitude of strong L-spaces of a given determinant}
\section{Finiteness results}  \label{sec: finiteness}

In this section, we prove Theorem \ref{t: finite}, which asserts that there exist finitely many rational homology spheres with bounded simultaneous trajectory number, and Theorem \ref{t: le8}, which classifies the strong L-spaces with determinant up to 8.  %We proceed via a sequence of lemmas.

\begin{lemma}
\label{l: decompose}
Let $H$ be a Heegaard diagram for a rational homology sphere.  Suppose that $H$ contains an attaching curve that has $m \ge 1$ intersection points with another attaching curve and at most one other intersection point.  Then there exists a sequence of isotopies and handleslides converting $H$ into a 1-extendible Heegaard diagram $H_1 \# H_2$ with $g(H_1) = 1$, $\abs{\SS(H_1)}=m$, and $\abs{\SS(H)} = \abs{\SS(H_1 \# H_2)}$.  If $H$ is strong, then so are $H_1$ and $H_2$.
%The same conclusions hold with the roles of $\alpha$ and $\beta$ curves exchanged.
\end{lemma}

\begin{proof}
Without loss of generality, assume that $\alpha_1$ has $m \ge 1$ intersection points with $\beta_1$.  If $\alpha_1$ meets only $\beta_1$, then the result follows directly from Lemma \ref{l: connected sum} and induction on the genus, as in the proof of Proposition \ref{p: 1-extendible}.  Suppose instead that $\alpha_1$ meets another curve $\beta_2$.  Label the intersection points along $\alpha_1$ consecutively by $p_1,\dots,p_m \in \beta_1$ and $q \in \beta_2$.  Perform $m$ consecutive handleslides of $\beta_1$ over $\beta_2$, guided along the oriented $\alpha_1$-arc from $p_i$ to $q$ for $i=m,\dots,1$ in turn.  Let $\beta_1'$ denote the resulting curve and $H'$ the resulting Heegaard diagram.  Observe that $\alpha_1$ has a single intersection point in $H'$.  Applying Lemma \ref{l: connected sum} and induction on the genus to $H'$ establishes the existence of handleslides and isotopies converting $H'$ into a 1-extendible Heegaard diagram of the required form with the property that $\abs{\SS(H')} = \abs{\SS(H_1 \conn H_2)}$.% $H_1 \# H_2$ with $g(H_1) = 1$ and $g(H_2) = g(H)-1$.

It remains to establish that $\abs{\SS(H')} = \abs{\SS(H)}$.  To do so, we exhibit a bijection between the perfect matchings in $G(H')$ and $G(H)$.  Let $G' \subset G(H)$ denote the subgraph obtained by removing the $m$ edges between $a_1$ and $b_1$.  Observe that $G(H')$ is constructed from $G'$ by inserting $m$ parallel edges between $a_i$ and $b_1$ for each edge between $a_i$ and $b_2$ in $G(H)$.  Thus, $G'$ is a common subgraph of $G(H)$ and $G(H')$.  In particular, we obtain a trivial bijection between the perfect matchings of $G(H)$ and $G(H')$ contained in this common subgraph.  Consider another perfect matching in $G(H)$.  Then it uses the $k\Th$ edge between $a_1$ and $b_1$ for some $1 \le k \le m$.  It also has an edge $e$ from $a_i$ to $b_2$ for some $i > 1$.  We construct a perfect matching in $G(H')$ by removing these two edges, putting in the edge from $a_1$ to $b_2$, and putting in the $k$-th new edge from $a_i$ to $b_1$ that corresponds to $e$.  It is clear that this construction sets up a bijection between the perfect matchings in $G(H)$ and $G(H')$ not contained in $G'$, and so completes the required bijection between perfect matchings in $G(H)$ and $G(H')$.
\end{proof}

For the remainder of this section, let $H_0$ denote the standard genus-$1$ Heegaard diagram of $\RP^3$. Recall that the minimum vertex degree of a graph $G$ is denoted $\delta(G)$, and $\delta(H) = \delta(G(H))$ for a Heegaard diagram $H$.

\begin{lemma}
\label{l: standard form}
Let $H$ be a Heegaard diagram for a rational homology sphere $Y$.  Then there exists a sequence of isotopies, handleslides, and destabilizations converting $H$ into a 1-extendible Heegaard diagram $H' = (\conn^n H_0) \conn H''$ such that $\abs{\SS(H)}=\abs{\SS(H')}$ and $\delta(H'') \ge 3$.  If $H$ is strong, then so are $H'$ and $H''$. \qed
\end{lemma}

\begin{proof}
We proceed by induction on the genus $g$ of $H$. The result is true if $g = 0$, so suppose that $g > 0$ and the result holds for Heegaard diagrams of genus less than $g$.
By Proposition \ref{p: 1-extendible}, we may assume that $H$ is 1-extendible.  If $\delta(H) \ge 3$, then the desired result follows at once.  Otherwise, $\delta(H) \le 2$.  Apply Lemma \ref{l: decompose} to an attaching curve in $H$ that contains at most two intersection points.  In the composite Heegaard diagram $H_1 \conn H_2$ guaranteed by Lemma \ref{l: decompose}, $H_1$ is the standard genus-$1$ diagram for $S^3$ or $\RP^3$.  Destabilize $H_1$ in the former case, and apply the induction hypothesis to $H_2$ to complete the induction step.
\end{proof}

\begin{lemma}
\label{l: finite graph}
There exist finitely many 1-extendible bipartite graphs $G$ with $\delta(G) \ge 3$ and a bounded number of perfect matchings.
\end{lemma}

\begin{proof}
Fix a natural number $d$ and suppose that $G$ is a 1-extendible bipartite graph that contains $d$ or fewer perfect matchings.  Let $n$ denote the number of vertices and $m$ the number of edges of $G$.  Then \cite[Theorem 7.6.2]{LovaszPlummerBook} establishes that
\begin{equation}
\label{e: matching bound}
d  \ge m - n + 2.
\end{equation}
In addition, $m \ge 3 n /2$, since $\delta(G) \ge 3$.
%\[
%\abs{E(G)} \ge \frac{\delta(G) \cdot \abs{V(G)}}{2} \ge \frac{3 \abs{V(G)}}{2}.
%\]
Applying \eqref{e: matching bound}, we obtain
\[
n \le 2(d-2) \quad \textup{and} \quad m \le d + n-2 \le 3(d-2).
\]
Since both of $m$ and $n$ are bounded in terms of $d$, $G$ is one of finitely many graphs.
\end{proof}

\begin{lemma}
\label{l: finite homeo}
For any graph $G$, there exist finitely many Heegaard diagrams $H$ with $G(H)=G$.
\end{lemma}

\begin{proof}
We may assume that $G$ has a bipartition $V(G) = A \sqcup B$ with $\abs{A}=\abs{B}=g$.
Label the edges of $G$ by $e_1, \dots, e_k$. Let $\alpha_1, \dots, \alpha_g, \beta_1, \dots, \beta_g$ be oriented copies of $S^1$, and let $N(\alpha_i)$ (resp.~$N(\beta_i)$) be the total space of a trivial $I$-bundle over $\alpha_i$ (resp.~$b_i$). Up to homeomorphism, there are finitely many ways to choose points $x_1, \dots, x_k, y_1, \dots, y_k$, where if $e_k$ is an edge connecting $a_{i_k}$ and $b_{j_k}$, then $x_k \in \alpha_{i_k}$ and $y_k \in \beta_{j_k}$. Given such a choice and a choice of $(\epsilon_1, \dots, \epsilon_k) \in \{\pm 1\}^k$, form a possibly disconnected, oriented surface-with-boundary $N$ by plumbing together the annuli $N(\alpha_i), N(\beta_j)$ so that $x_k$ and $y_k$ are identified and $\alpha_{i_k}$ and $\beta_{j_k}$ meet with sign $\epsilon_k$ at that point. There are finitely many ways to glue surfaces-with-boundary to $N$ along their common boundaries to obtain a closed, oriented surface $S$ of genus $g$. Thus, up to homeomorphism, there are finitely many tuples $(S,\alpha,\beta)$, where $S$ is a closed, oriented surface of genus $g$ and $\alpha$ and $\beta$ are $g$-tuples of curves in $S$ whose intersection graph is given by $G$, and in particular finitely many such Heegaard diagrams.
\end{proof}

\begin{proof}[Proof of Theorem \ref{t: finite}]
Fix a natural number $d$ and suppose that $Y$ is a rational homology sphere with $M(Y) \le d$.  By Lemma \ref{l: standard form}, $Y$ has a 1-extendible Heegaard diagram $H = (\conn^n H_0) \conn H'$ with $M(Y) = \abs{\SS(H)}=2^n \abs{\SS(H')}$ and $\delta(H') \ge 3$.  Thus, $n \le \log_2 d$ and $G(H')$ contains at most $d$ perfect matchings.  By Lemma \ref{l: finite graph}, there exist finitely many possibilities for $G(H')$, so by Lemma \ref{l: finite homeo}, there exist finitely many possibilities in turn for $H'$.  Therefore, there exist finitely many possibilities for $H$ and hence for $Y$, as required.
\end{proof}

We now turn to the proof of Theorem \ref{t: le8}, which asserts that every strong L-space with determinant $\le 8$ is the branched double cover of an alternating link in $S^3$. In the proof, we apply an estimate on the number of perfect matchings in a cubic bipartite graph due to Voorhoeve \cite{VoorhoevePermanents}.  A graph is {\em cubic} if every vertex has degree 3.
The proof of \cite[Theorem 8.1.7]{LovaszPlummerBook} establishes the following version of Voorhoeve's result:
\begin{lemma}
\label{l: voorhoeve}
Define $h\co \bZ^+ \to \bZ^+$ recursively by $h(1) = 2$ and $h(g) = \lceil 4h(g-1)/3 \rceil$, and define $f\co \bZ^+ \to \bZ^+$ by $f(g) = \lceil 3 h(g) / 2 \rceil$.  Then a cubic bipartite graph on $2g$ vertices contains at least $f(g)$ perfect matchings. In particular, a cubic bipartite graph with 8 or more vertices contains at least $f(4)=9$ perfect matchings.\qed
\end{lemma}

\begin{proof}[Proof of Theorem \ref{t: le8}]
Choose a 1-extendible, strong Heegaard diagram $H$ of minimum genus $g$ presenting $Y$.  If $g=1$, then $Y$ is the branched double cover of an alternating two-bridge link.  If $g = 2$, then the result follows from Theorem \ref{t: genus2} (proven in Section \ref{sec: genus2} and independent of this result).  If there exists a sequence of isotopies and handleslides converting $H$ into a connected sum of strong Heegaard diagrams $H_1 \conn H_2$, then by induction on $g$, $H_i$ presents $\Sigma(L_i)$ for some non-split alternating link $L_i$, and then $Y \cong \Sigma(L_1 \conn L_2)$ exhibits $Y$ in the stated form.

Thus, unless the desired conclusion holds, we may assume henceforth that $g \ge 3$ and that the hypothesis of Lemma \ref{l: decompose} does not hold; in particular, every vertex of $G(H)$ has degree at least $3$, and no vertex of degree $3$ is incident with parallel edges (two edges with the same pair of endpoints).  We seek a contradiction to these conditions under the assumption that $\det(Y) \le 8$.

First, suppose that $g=3$.  Since $M(H)$ is a P\'olya matrix, it contains at least one zero entry, as noted in Section \ref{sec: preliminaries}.  Since the condition of Lemma \ref{l: decompose} does not hold, no row or column of $M(H)$ contains more than one zero, and if a row or column contains a zero, then its non-zero entries are at least two in absolute value. It follows that up to permuting its rows and columns, $M(H)$ {\em dominates} one of the following three matrices, in the sense that the absolute values of its entries bound from above those of the corresponding matrix:
\[
\left( \begin{matrix} 0 & 2 & 2 \\ 2 & 0 & 2 \\ 2 & 2 & 0\end{matrix} \right), \left( \begin{matrix} 1 & 2 & 2 \\ 2 & 0 & 2 \\ 2 & 2 & 0\end{matrix} \right), \left( \begin{matrix} 1 & 1 & 2 \\ 1 & 1 & 2 \\ 2 & 2 & 0\end{matrix} \right).
\]
If a P\'olya matrix $P$ dominates a non-negative matrix $N$, then $\abs{\det(P)} = \per(\abs{P}) \ge \per(N)$.  It follows that
\[
\abs{\SS(H)} = \abs{\det(M(H))} \ge \min \{16,20,16\} = 16,
\]
a contradiction.

Next, suppose that $g = 4$.  We argue that $G(H)$ contains a cubic subgraph on 8 vertices, which by Lemma \ref{l: voorhoeve} is a contradiction.  Inequality \eqref{e: matching bound} implies that $G(H)$ contains at most 14 edges, and by hypothesis it contains at least 12 edges.  If it has 12 edges, then $G(H)$ is the desired subgraph.  If it has 13 edges, then there exists a unique vertex of degree 4 in each of $A$ and $B$.  Since no vertex of degree $3$ is incident with parallel edges, it follows that there exists an edge $e$ between the vertices of degree $4$, and then $G(H) \minus e$ is the desired subgraph.  If it has 14 edges, then the degree sequences of vertices in $A$ and in $B$ belong to $\{(3,3,3,5),(3,3,4,4)\}$.  Again, since no vertex of degree $3$ is incident with parallel edges, every vertex of degree $4$ or more is incident with another such vertex; furthermore, if there are two vertices of degree $5$, then there are parallel edges between them.  In any case, we can locate edges $e,f$ so that $G(H) \minus \{e,f\}$ is the desired subgraph.

Lastly, suppose that $g \ge 5$.  $G(H)$ cannot contain a cubic subgraph on $2g$ vertices by Lemma \ref{l: voorhoeve}, because then it would contain more than 8 perfect matchings, a contradiction.  Thus, $G(H)$ contains at least $3g+1$ edges.  Since $|\SS(H)| \le 8$, inequality \eqref{e: matching bound} implies that $g \le 5$.  Consequently, $g=5$, there are vertices $a \in A$ and $b \in B$ of degree four, and all other vertices have degree three and are not incident with parallel edges.  If there exists an edge $e = (a,b)$, then $G(H) \minus e$ is a cubic subgraph on 10 vertices, a contradiction.  Therefore, $a$ and $b$ are non-adjacent and $G(H)$ has no parallel edges.  Thus, $G(H) \minus \{a,b\}$ is a 2-regular bipartite graph.  It is easy to see in this case that any two-edge matching that uses vertices $a$ and $b$ extends to a perfect matching of $G(H)$.  However, there are 16 such two-edge matchings, whereas $G(H)$ has at most 8 perfect matchings, a contradiction. This concludes the proof that $Y \cong \Sigma(L)$ for some alternating link $L \subset S^3$.

Finally, the determinant of an alternating link is greater than or equal to its crossing number, with equality only for $(2,n)$-torus links \cite{CrowellNonalternating}.  It follows that $L$ has at most seven crossings or is the $(2,8)$-torus link. The knot tables indicate that all prime, alternating links through seven crossings with determinant $\le 8$ are two-bridge links.  Therefore, $L$ is a connected sum of two-bridge links, and $\Sigma(L)$ is a connected sum of lens spaces.
\end{proof}

\section{Strong diagrams of genus two} \label{sec: genus2}

The purpose of this section is to prove Theorem \ref{t: genus2}, describing all strong L-spaces admitting strong Heegaard diagrams of genus $2$.

\subsection{Coherent multicurves in an annulus}

We begin with some technical but elementary statements concerning curves in an annulus that will enable us to recognize certain standard configurations within a strong Heegaard diagram.

\begin{construction} \label{const: annulus}
Fix orientations on the circle $S^1$ and the interval $I$. Let $A = S^1 \times I$ denote the annulus, equipped with the product orientation. For integers $p_1, q_1, p_2, q_2$, let $\bm\gamma(p_1,q_1,p_2,q_2)$ denote the pair of oriented multicurves $(\gamma_{p_1,q_1}, \gamma_{p_2,q_2})$ obtained as follows: choose distinct $s_1 \ne s_2 \in S^1$ and $t_1 < t_2 \in I$, and set
\[
\gamma_{p_i,q_i} = p_i(S^1 \times \{t_i\}) + q_i(\{s_i\} \times I),
\]
where we do not allow the parallel copies of $\pm S^1 \times \{t_1\}$ (resp.~$\pm \{s_1\} \times I$) to overlap or interleave with the parallel copies of $\pm S^1 \times \{t_2\}$ (resp.~$\pm \{s_2\} \times I$), and we require $\gamma_{p_1,q_1}$ and $\gamma_{p_2,q_2}$ to meet transversally. (See Figure \ref{fig:annuluslemma}.)
\end{construction}

\begin{figure}
\labellist
 \pinlabel {{\color{red} $\gamma_{-2,3}$}} [Bl] at 135 203
 \pinlabel {{\color{blue} $\gamma_{3,2}$}} [Bl] at 135 175
 \pinlabel {{\color{red} $\gamma_{2,3}$}} [Bl] at 407 203
 \pinlabel {{\color{blue} $\gamma_{3,2}$}} [Bl] at 407 175
\endlabellist
\includegraphics[scale=0.75]{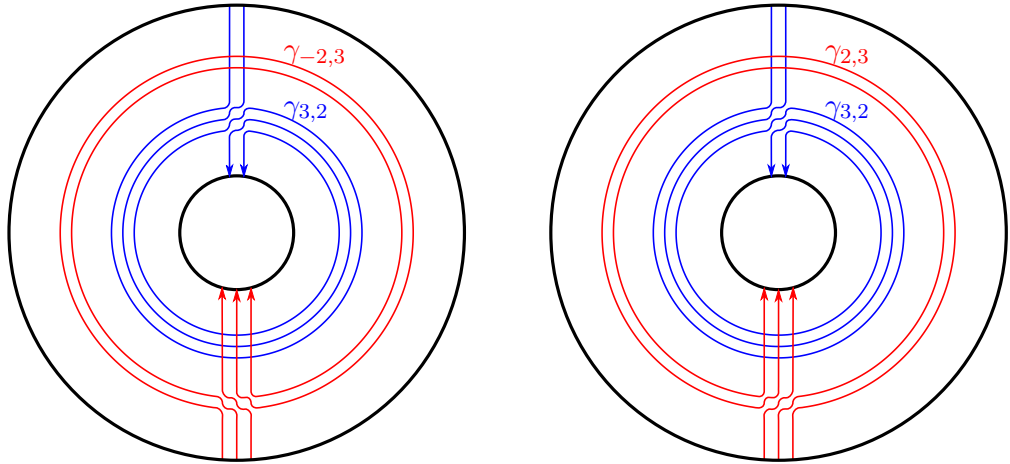}
\caption{The configurations $\bm\gamma(-2,3,3,2)$ and $\bm\gamma(2,3,3,2)$ in $S^1 \times I$. Note that the former is in minimal position while the latter is not.} \label{fig:annuluslemma}
\end{figure}

In Construction \ref{const: annulus}, we have
\[
\gamma_{p_1,q_1} \cdot \gamma_{p_2,q_2} = p_1 q_2 - q_1 p_2 \quad \text{and} \quad \abs{\gamma_{p_1,q_1} \cap \gamma_{p_2,q_2}} = \abs{p_1 q_2} + \abs{q_1 p_2}.
\]
In particular, the pair $(\gamma_{p_1,q_1}, \gamma_{p_2,q_2})$ is in minimal position except when $p_1 q_2$ and $q_1 p_2$ are either both positive or both negative.

%The mapping class group of $A$ (fixing the boundary) is isomorphic to $\Z$, generated by a Dehn twist. A right-handed Dehn twist takes $\bm\gamma(p_1,q_1, p_2,q_2)$ to $\bm\gamma(p_1+q_1,q_1, p_2+q_2,q_2)$.

The ambient isotopy class (fixing a point on each component of $\partial A$) of each multicurve $\gamma_{p_i,q_i}$ depends only on $(p_i,q_i)$, but the full configuration $\bm\gamma(p_1,q_1, p_2,q_2)$ depends a priori on more than just the isotopy classes of the individual multicurves. However, the following lemma says that when the pair $(\gamma_{p_1,q_1},\gamma_{p_2,q_2})$ is in minimal position, the configuration is uniquely characterized up to homeomorphism.

\begin{lemma}
\label{l: annulus}
Let $r_1, \dots, r_k, s_1, \dots, s_l$ be points in $S^1$, ordered cyclically according to the standard orientation of $S^1$. Suppose that we are given properly embedded, oriented multicurves $\alpha = \alpha_1 \cup \dots \cup \alpha_k$ and $\beta = \beta_1 \cup \dots \cup \beta_l$ in $A$ satisfying the following properties:
\begin{itemize}
\item
For some fixed $[a] \in \Z/k$, for each $i \in \{1, \dots, k\}$, $\alpha_i$ is a path from $(r_i, 0)$ to $(r_{i+a},1)$ (indices modulo $k$), and any two paths $\alpha_i, \alpha_{i'}$ are disjoint.
\item
For some fixed $[b] \in \Z/l$, for each $j \in \{1, \dots, l\}$, $\beta_j$ is a path from $(s_j, 0)$ to $(s_{j+b},1)$ (indices modulo $l$), and any two paths $\beta_j, \beta_{j'}$ are disjoint.
\item
The multicurves $\alpha$ and $\beta$ intersect transversally and coherently.
\end{itemize}
Then for some integers $a$ and $b$ restricting to the given classes $[a]$ and $[b]$ and satisfying $\abs{al-bk} = \abs{\alpha \cap \beta}$, there is a homeomorphism of $A$ taking $(\alpha,\beta)$ to $\bm\gamma(a,k,b,l)$.
\end{lemma}

\begin{proof}
For concreteness, assume that the intersection points in $\alpha_i \cap \beta_j$ all have positive sign; the other case is analogous.

Let $a$ be the representative of $[a]$ in $\{0, \dots, k-1\}$. We may identify $A$ with the quotient of the rectangle $R = [0,k] \times [0,1]$ by the relation $(0,y) \sim (k,y)$ for all $y \in [0,1]$, such that $\alpha_i$ is the image of $\{i\} \times [0,1]$; denote the projection map $\pi\co R \to A$. For $i=1, \dots, k$, note that $\pi(i,0) = (r_i,0)$, while $\pi(i,1) = (r_{i+a},1)$ (indices modulo $k$). Therefore, $\beta$ is the image of a collection of oriented arcs $\tilde \beta \subset R$ with endpoints on
\[
\left( [0,1] \times \{0\} \right) \cup \left( [k-a,k-a+1] \times \{1\} \right) \cup \left( \{0,k\} \times [0,1] \right).
\]

For $i=1, \dots, k$, let $R_i$ denote the square $[i-1,i] \times [0,1]$. Since $\alpha$ meets $\beta$ positively, any segment of $\tilde \beta \cap R_i$ (with its inherited orientation) must enter $R_i$ through $\{i\} \times [0,1]$ (or through $[0,1] \times \{0\}$ if $i = 1$) and exit through $\{i-1\} \times [0,1]$ (or through $[a-1,a] \times \{1\}$ if $i=a$). After an ambient isotopy of $R_i$, we may arrange that $\tilde \beta \cap R_i$ is a union of line segments of negative slope. Next, after an ambient isotopy of $R$ that leaves the first coordinate fixed and reparametrizes the second coordinate respecting $\sim$, we may arrange that $\tilde \beta$ is union of oriented line segments of negative slope beginning
on
\[
\left( [0,1] \times \{0\} \right) \cup \left( \{k\} \times [0,1] \right)
\]
and ending on
\[
\left( [k-a,k-a+1] \times \{1\} \right) \cup \left( \{0\} \times [0,1] \right).
\]

Let $m = \abs{\tilde \beta \cap (\{0\} \times [0,1])} = \abs{\tilde \beta \cap (\{k\} \times [0,1])}$. Let $c$ be the number of segments of $\tilde \beta$ that begin on $\{k\} \times [0,1]$ and end on $[k-a,k-a+1] \times \{1\}$. After another ambient isotopy of $R$ that leaves the first coordinate fixed and reparametrizes the second coordinate respecting $\sim$, we may arrange that $\tilde \beta$ is a union of line segments of negative slope as above, with the additional properties that:
\begin{itemize}
\item $m-c$ segments end on $\{0\} \times [0,\frac12]$, and $c$ segments end on $\{0\} \times [\frac12,1]$;
\item $m-c$ segments begin on $\{k\} \times [0,\frac12]$, and $c$ segments end on $\{k\} \times [\frac12,1]$; and
\item $\tilde \beta \cap ([0,k] \times \{\frac12\})$ consists of $l$ points, all in the interior of $R_1$.
\end{itemize}

We may now reparametrize $R/{\sim}$ by the transformation
\[
f(x,y) =
\begin{cases}
(x,y) & y \le \frac12 \\
(x + 2ay-a,y) & y \ge \frac12.
\end{cases}
\]
After this reparametrization, $\beta \cap ([0,k] \times [\frac12,1])/{\sim}$ can be arranged to be vertical, while $\beta \cap ([0,k] \times [\frac12,1])/{\sim}$ winds with positive slope. From here, it is not hard to identify $(\alpha,\beta)$ with $\bm\gamma(a, k, c-m, l)$.
%\smargin{Say more here?}
\end{proof}

\begin{construction} \label{const: plumbing}
Returning to the notation in Construction \ref{const: annulus}, suppose that $\gcd(p_1,q_1) = \gcd(p_2,q_2) = 1$. Decompose $S^1$ as a union of two arcs $e_1 \cup e_2$ so that $\partial \gamma_{p_i,q_i} \subset e_i \times \{0,1\}$. Let $R_1$ and $R_2$ be $1$-handles (i.e., copies of $I \times I$), and let $N$ be the oriented surface obtained by attaching $R_1$ and $R_2$ to $A$ along $e_1 \times \{0,1\}$ and $e_2 \times \{0,1\}$, respectively. For $i=1,2$, by attaching $q_i$ parallel cores of $R_i$ to $\gamma_{p_i,q_i}$, we obtain a closed curve $\eta_{p_i/q_i}$. We may canonically orient these curves by requiring $q_1,q_2 \ge 0$ and taking the induced orientation.
\end{construction} 

\def\bconn{\mathbin{\natural}}

\subsection{Conventions for rational tangles} \label{ss:tangles}

We briefly review some basic facts and conventions about rational tangles and their branched double covers; for further information, see \cite[Section 4]{GordonPCMI} and \cite[Chapter 8]{Cromwell:book}).

Consider the sphere $S^2 = \partial B^3$, and fix an equatorial $S^1 \subset S^2$. Let $Q = \{NE, NW, SW, SE\}$ be a set of four points on this $S^1$, ordered cyclically. By a slight abuse of notation, we shall denote by $\mu$ either the (unoriented) arc of $S^1$ joining $SW$ and $NW$ or the arc joining $SE$ and $NE$, and denote by $\lambda$ either the arc joining $NW$ and $NE$ or the arc joining $SW$ and $SE$. For $p/q \in \Q \cup \{1/0\}$, let $R(p/q)$ denote the $p/q$ rational tangle in $B^3$, with endpoints on $Q$; our convention is that the components of $R(1/0)$ (resp.~$R(0/1)$) are pushoffs of the two choices of $\mu$ (resp.~$\lambda$). Note that $R(p/q)$ can be represented with an alternating diagram such that the first crossing encountered when entering from $SW$ or $NE$ is an undercrossing, and the first crossing encountered when entering from $NW$ or $SE$ is an overcrossing, precisely when $p/q \ge 0$.

The branched double cover $\Sigma(S^2,Q)$ is a torus, and $\Sigma(B^3,R(p/q))$ is a solid torus. Let $\tilde \mu$ (resp.~$\tilde \lambda$) be the preimage of $\mu$ (resp.~$\lambda$), which up to isotopy does not depend on the choice above. We orient $\tilde\mu$ and $\tilde\lambda$ such that $\tilde\mu \cdot \tilde\lambda = -1$ when $\Sigma(S^2,Q)$ is oriented as the boundary of $\Sigma(B^3,R(p/q))$. As explained in \cite[Section 4]{GordonPCMI}, for any $p/q \in \Q \cup \{1/0\}$, the curve $p \tilde \mu + q \tilde \lambda$ bounds a compressing disk in $\Sigma(B^3,R(p/q))$.

\subsection{A construction of genus-2 Heegaard diagrams} \label{ss: genus2diagram}

%\begin{figure}[htb!]
%\centering
%\includegraphics[width=4in]{tangle}
%\put(-205,55){\textcolor{red}{$\mu$}}
%\put(-240,85){\textcolor{blue}{$\lambda$}}
%\put(-100,55){\textcolor{red}{$\widetilde{\mu}$}}
%\put(-30,80){\textcolor{blue}{$\widetilde{\lambda}$}}
%\caption{At left is an oriented sphere with four marked points and two arcs $\mu$ and $\lambda$.  At right is the double cover of the sphere branched along the four points, along with the lifted orientation and lifted curves $\widetilde{\mu}$ and $\widetilde{\lambda}$.  The curves are oriented so that $\widetilde{\mu} \cdot \widetilde{\lambda} = 1$.}  \label{f:tangle}
%\end{figure}
%

%\begin{figure}[htb!]
%\centering
%\includegraphics[width=4in]{tangle3}
%\put(-258,43){\Large $p$}
%\put(-247,32){\Large $q$}
%\put(-100,55){\Large $p$}
%\put(-30,80){\Large $q$}
%\caption{At left is graphical notation for the standard alternating diagram of the rational tangle $(B^3,T(p/q))$, $p/q \in \bQ \cup \{1/0\}$.  The ball $B^3$ is oriented as a subspace of $S^3$, which has a fixed orientation, and $\del B^3$ is oriented as its boundary.  The branched double cover $\Sigma(B^3,T(p/q))$ is homeomorphic to a solid torus.  At right is its boundary along with a labeled pair of curves representing $p \widetilde{\mu} + q \widetilde{\lambda}$, which bounds a compressing disk in $\Sigma(B^3,T(p/q))$.}  \label{f:tangle3}
%\end{figure}

%The tangle diagram for $(B^3,T(p/q))$ is chosen to be alternating.

As before, for $a_1, a_2, a_3, a_4 \in \Q \cup \{1/0\}$, let $L(a_1, a_2, a_3, a_4)$ denote the link given by the diagram in Figure \ref{f:branchset}, where the four balls are filled in with rational tangles following the conventions described above. If $a_1, a_2, a_3, a_4$ have the same sign, then the diagram obtained in this manner (using alternating diagrams for the rational tangles) is alternating. We now describe a construction of a Heegaard diagram for $\Sigma(L(a_1,a_2,a_3,a_4))$.

\begin{construction} \label{const: H(a1a2a3a4)}
For $a_i = p_i/q_i \in \bQ \cup \{ 1/0 \}$, $i=1,\dots,4$, let $H(a_1,a_2,a_3,a_4)$ denote the Heegaard diagram shown in Figure \ref{f: genus2}. Formally, we take two copies of the configuration from Construction \ref{const: plumbing}, denoted $(N, \eta_{-q_1/p_1}, \eta_{q_2/p_2})$ and $(N',\eta_{p_3/q_3}, \eta_{-p_4/q_4})$, and form the union $S = (N \cup N') / {\sim}$, where $R_1 \subset N$ is identified with $R_2 \subset N'$, and $R_2 \subset N$ is identified with $R_1 \subset N'$, each via a $90^\circ$ rotation of $I \times I$; we define
\begin{align*}
\alpha_1 &= \eta_{-q_1/p_1} \subset N & \beta_1 &= \eta_{q_2/p_2} \subset N \\
\alpha_2 &= \eta_{p_3/q_3} \subset N' & \beta_2 &= \eta_{-p_4/q_4} \subset N'.
\end{align*}
\end{construction}

%\smargin{Add dotted lines to Figure \ref{f: genus2} showing this decomposition?}

%\begin{figure}%[htb!]
%\centering
%\includegraphics[width=5in]{genus2diagram}
%\put(-350,140){\textcolor{red}{$p_2$}}
%\put(-260,140){\textcolor{red}{$q_4$}}
%\put(-110,140){\textcolor{blue}{$p_1$}}
%\put(-20,140){\textcolor{blue}{$q_3$}}
%\put(-305,175){\textcolor{red}{$q_2$}}
%\put(-65,25){\textcolor{red}{$p_4$}}
%\put(-305,25){\textcolor{blue}{$q_1$}}
%\put(-65,175){\textcolor{blue}{$p_3$}}
%%\put(-325,175){\textcolor{red}{$\alpha_1$}}
%%\put(-50,25){\textcolor{red}{$\alpha_2$}}
%%\put(-325,25){\textcolor{blue}{$ \beta_1$}}
%%\put(-50,175){\textcolor{blue}{$\beta_2$}}
%\caption{Template for the Heegaard diagram $H(a_1,a_2,a_3,a_4)$, $a_i = p_i/q_i \in \bQ\cup\{1/0\}$, $i=1,\dots,4$.}  \label{f: genus2}
%\end{figure}

\begin{figure}%[htb!]
\labellist
 \pinlabel {{\color{red} $\alpha_1$}} [tr] at 90 12
 \pinlabel {{\color{darkred} $\alpha_2$}} [bl] at 258 174
 \pinlabel {{\color{blue} $\beta_1$}} [br] at 90 174
 \pinlabel {{\color{darkblue} $\beta_2$}} [tl] at 258 12
 \pinlabel {{\color{red} $p_1$}} [tl] at 239 62
 \pinlabel {{\color{red} $q_1$}} [tl] at 203 40
 \pinlabel {{\color{blue} $p_2$}} [br] at 20 124
 \pinlabel {{\color{blue} $q_2$}} [br] at 60 150
 \pinlabel {{\color{darkred} $p_3$}} [bl] at 288 150
 \pinlabel {{\color{darkred} $q_3$}} [bl] at 328 124
 \pinlabel {{\color{darkblue} $p_4$}} [tr] at 145 40
 \pinlabel {{\color{darkblue} $q_4$}} [tr] at 109 62
 \endlabellist
\includegraphics[width=5in]{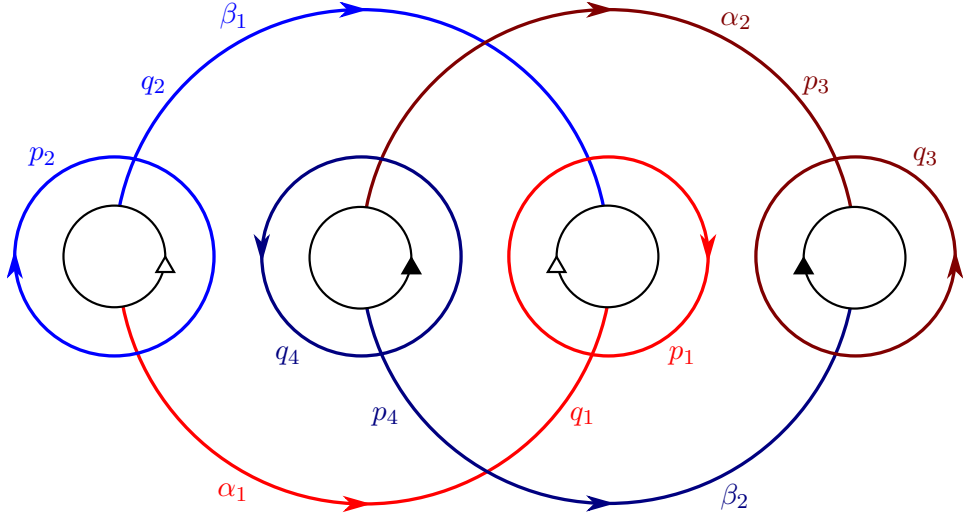}
\caption{Template for the Heegaard diagram $H(a_1,a_2,a_3,a_4)$, $a_i = p_i/q_i \in \bQ\cup\{1/0\}$, $i=1,\dots,4$.}  \label{f: genus2}
\end{figure}

Observe that
\[
M(H(a_1,a_2,a_3,a_4)) =
\begin{pmatrix}
-p_1q_2 - q_1p_2 & -q_1p_4 \\
-q_2p_3 & p_3q_4 + q_3p_4
\end{pmatrix},
\]
so
\[
\det M(a_1,a_2,a_3,a_4) = -p_1q_2 q_3 p_4 - q_1 p_2 q_3 p_4 - q_1 q_2 p_3 p_4,
\]
while
\[
\abs{\SS(H(a_1,a_2,a_3,a_4))} = \abs{p_1q_2 q_3 p_4} + \abs{q_1 p_2 q_3 p_4} + \abs{q_1 q_2 p_3 p_4}.
\]
It is easy to see that $H(a_1,a_2,a_3,a_4)$ is strong if and only if $a_1,a_2,a_3,a_4$ all have the same sign. Additionally, note that when $a_1 = 1/0$ and $a_3= 0/1$ or when $a_2 = 1/0$ and $a_4 = 0/1$, 
$H(a_1,a_2,a_3,a_4)$ is a connected sum of genus-$1$ strong diagrams.

\begin{proposition} \label{p: hd}
For $a_1,a_2,a_3,a_4 \in \bQ\cup\{1/0\}$, the Heegaard diagram $H(a_1,a_2,a_3,a_4)$ presents the manifold $\Sigma(L(a_1,a_2,a_3,a_4))$.
\end{proposition}

%\begin{figure}[htb!]
%\centering
%\includegraphics[width=2in]{branchset}
%\put(-129,58){\Large $a_1$}
%\put(-82,58){\Large $a_2$}
%\put(-21,82){\Large $a_3$}
%\put(-21,35){\Large $a_4$}
%\caption{Diagram of the link $L(a_1,a_2,a_3,a_4)$, $a_1,a_2,a_3,a_4 \in \bQ\cup\{1/0\}$.}  \label{f:branchset}
%\end{figure}
%
%Note that the diagram is alternating iff all $a_i$ have the same sign, where, by convention, $0/1$ and $1/0$ have both signs.  Setting all $a_i$ equal to one of $\pm 1$ results in the minimal diagram $D$ of the figure eight knot, so the diagram can be regarded as substituting arbitrary rational tangles for the crossings in $D$.

%Definition of the link $L(a_1,a_2,a_3,a_4)$, $a_1,a_2,a_3,a_4 \in \bQ^+$.  The alternating condition.

%%%
%%%
%%%

%\subsection{Proof of Theorem \ref{t:hd}} \hspace{1in}

\begin{proof}
Figure \ref{f:branchset2} depicts the link $L = L(a_1,a_2,a_3,a_4)$ along with some additional decoration that we will use in order to produce a Heegaard decomposition of $\Sigma(L)$.  We will then show that $H(a_1,a_2,a_3,a_4)$ presents this decomposition. The red curve $C$ in the projection plane is the cross-section of a sphere $P$ that meets $L$ in six points $x_1,\dots,x_6$.    The regions cut out by $C$ in the projection plane are the cross-sections of the balls cut out by $P$ in $S^3$.  The two additional arcs in blue and yellow are cross-sections of equatorial disks $D_1$ and $D_2$ for these balls. The space $S^3 \minus (P \cup D_1 \cup D_2)$ consists of four open balls whose closures we denote by $A_1,\dots,A_4$; note that $(A_i,A_i \cap L)$ is the rational tangle $(B^3,R(a_i))$.  Orient $P$ as the boundary of $A_1 \cup A_3$. The double cover $V_i = \Sigma(A_i,A_i \cap L)$ is then a solid torus, and $\Sigma(L) = V_1 \cup \dots \cup V_4$.

Since $A_1 \cap A_3 = D_1$, $A_2 \cap A_4 = D_2$, and $D_1$ and $D_2$ each meet $L$ in a single point, it follows that $U_1 =\Sigma(A_1 \cup A_3, (A_1 \cup A_3) \cap L)$ and $U_2 = \Sigma(A_2 \cup A_4, (A_2 \cup A_4) \cap L)$ are each genus-$2$ handlebodies. Specifically, $U_1 = V_1 \bconn V_3$ and $U_2 = V_2 \bconn V_4$, where $\bconn$ denotes the boundary connected sum. Thus, $U_1 \cup_S U_2$ is a genus two Heegaard decomposition of $\Sigma(L)$, where $S = \del U_1 = -\del U_2 = \Sigma(P, P \cap L)$.

\begin{figure}[t]
\centering
\includegraphics[width=3in]{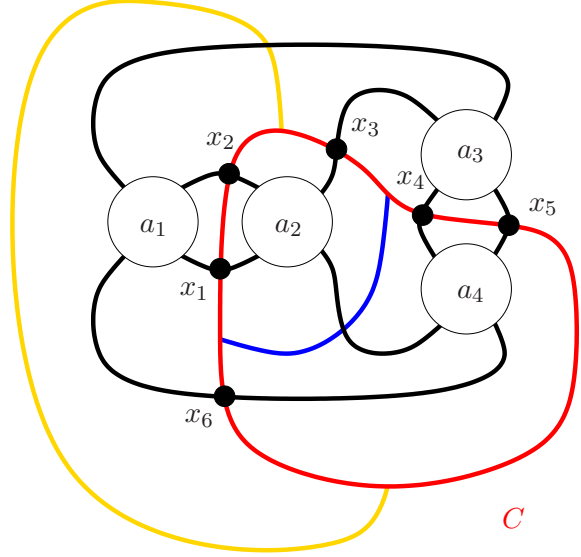}
\put(-167,122){$a_1$}
\put(-116,122){$a_2$}
\put(-47,150){$a_3$}
\put(-47,97){$a_4$}
\put(-152,98){$x_1$}
\put(-142,155){$x_2$}
\put(-87,160){$x_3$}
\put(-70,140){$x_4$}
\put(-20,130){$x_5$}
\put(-150,50){$x_6$}
\put(-30,10){\textcolor{red}{$C$}}
\caption{Decomposing $(S^3, L(a_1,a_2,a_3,a_4))$.}  \label{f:branchset2}
\end{figure}

For $i=1,\dots,6$, let $C_i$ denote the arc of $C$ that runs clockwise from  $x_i$ to $x_{i+1}$ in Figure \ref{f:branchset2} (subscripts mod $6$).  Let $N_i \subset P$ denote a small regular neighborhood of $C_i$ and $N = N_1 \cup \cdots \cup N_6$. 
Let $\tilde x_i$, $\gamma_i$, and $\gamma$ denote the respective preimages of $x_i$, $C_i$, and $C$ in $S \subset \Sigma(L)$, $i=1,\dots,6$.
%Let $\gamma$ denote the preimage of $C$ in $S \subset \Sigma(L)$, and let $\tilde x_1,\dots,\tilde x_6$ and $\gamma_1, \dots, \gamma_6$ denote the respective preimages of $x_1,\dots,x_6$ and $C_1, \dots, C_6$. 
Then $\tilde N_i = \Sigma(N_i,\{x_i,x_{i+1}\})$ is an annulus with core $\gamma_i$, and the union $\tilde N = \Sigma(N, \{x_1,\dots,x_6\})$ is a cyclic plumbing of these annuli. We may orient the curves $\gamma_i$ such that $\gamma_i \cdot \gamma_{i+1} = 1$ for $i=1,\dots,6$. Now $S = \Sigma(P,\{x_1,\dots,x_6\})$ results from gluing $A$ to the double cover of $P \minus N$, which consists of four disks; it follows that $\gamma \subset S$ appears as shown in Figure \ref{f: genus2template}.

\begin{figure}[t]%[htb!]
%\centering
%\includegraphics[width=5in]{genus2improved}
%\put(-60,170){\textcolor{red}{$\gamma$}}
%\put(-210,125){\textcolor{green}{$\gamma_4'$}}
%\put(-170,125){\textcolor{green}{$\gamma_1'$}}
%\put(-340,42){$\tilde x_1$}
%\put(-340,135){$\tilde x_2$}
%\put(-185,180){$\tilde x_3$}
%\put(-35,135){$\tilde x_4$}
%\put(-35,42){$\tilde x_5$}
%\put(-185,-8){$\tilde x_6$}
\labellist
 \pinlabel $\bullet$ at 46 128
 \pinlabel $\bullet$ at 46 57
 \pinlabel $\bullet$ at 175 171
 \pinlabel $\bullet$ at 175 14
 \pinlabel $\bullet$ at 302 128
 \pinlabel $\bullet$ at 302 57
 \pinlabel {$\tilde x_1$} at 43 49
 \pinlabel {$\tilde x_2$} at 42 137
 \pinlabel {$\tilde x_3$} at 175 181
 \pinlabel {$\tilde x_4$} at 306 137
 \pinlabel {$\tilde x_5$} at 306 49
 \pinlabel {$\tilde x_6$} at 175 5
 \pinlabel {{\color{red} $\gamma_1$}} [r] at 3 93
 \pinlabel {{\color{red} $\gamma_2$}} [br] at 90 174
 \pinlabel {{\color{red} $\gamma_3$}} [bl] at 258 174
 \pinlabel {{\color{red} $\gamma_4$}} [l] at 345 93
 \pinlabel {{\color{red} $\gamma_5$}} [tl] at 258 12
 \pinlabel {{\color{red} $\gamma_6$}} [tr] at 90 12
 \pinlabel {{\color{green} $\gamma_1'$}} [tl] at 239 62
 \pinlabel {{\color{green} $\gamma_4'$}} [tr] at 109 62
\endlabellist
\includegraphics[width=5in]{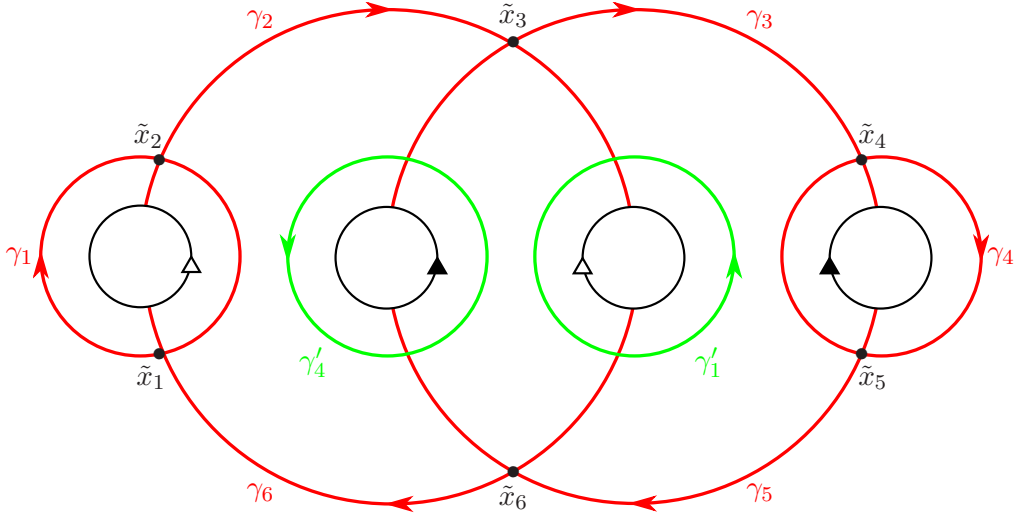}
\caption{The surface $S = \Sigma(P, \{x_1,\dots,x_6\})$ along with the branch points $\tilde x_1,\dots,\tilde x_6$, the 1-complex $\gamma$ in red, and the curves $\gamma_1'$, $\gamma_4'$ in green.  The curve $\gamma_i$ consists of the pair of arcs of $\gamma$ with endpoints $\tilde x_i$ and $\tilde x_{i+1}$.  The curves $\gamma_1'$ and $\gamma_4'$ are parallel copies of $\gamma_1$ and $\gamma_4$, respectively.}  \label{f: genus2template}
\end{figure}

%{\em Completing the Heegaard diagram.}

%Finally, Figure \ref{f: genus2} displays a Heegaard diagram for $\Sigma(L(a_1,a_2,a_3,a_4))$.
%in the genus two Heegaard splitting of $\Sigma(L)$ both decompose into the boundary sum of two solid tori: $U = \Sigma(A_1,A_1 \cap L) \bconn \Sigma(A_3,A_3 \cap L)$ and $V = \Sigma(A_2,A_2 \cap L) \bconn \Sigma(A_4,A_4 \cap L)$.
%Using the rational tangle decomposition of $L$, we can describe a curve on the boundary of each $\Sigma(A_i,A_i \cap L)$ that bounds a compressing disk for it.
Define
\begin{align*}
\mu_1 &= C_1, &  \mu_2 &= C_1, & \mu_3 &= C_3, & \mu_4 &= C_5, \\
\lambda_1 &= C_6, & \lambda_2 &= C_2, & \lambda_3 &= C_4, & \lambda_4 &= C_4,
\end{align*}
and observe that $\mu_i$ and $\lambda_i$ are framing arcs for the rational tangle $(A_i, L \cap A_i)$, as described above. Let $\tilde \mu_i, \tilde \lambda_i \subset S \cap \partial V_i$ denote the preimages of these curves ($i=1, \dots, 4$), oriented as follows:
\begin{align*}
\tilde\mu_1 &= \gamma_1, &  \tilde\mu_2 &= \gamma_1, & \tilde\mu_3 &= \gamma_3, & \tilde\mu_4 &= -\gamma_5, \\
\tilde\lambda_1 &= \gamma_6, & \tilde\lambda_2 &= \gamma_2, & \tilde\lambda_3 &= -\gamma_4, & \tilde\lambda_4 &= \gamma_4.
\end{align*}
With respect to the orientation of $S$ as $\partial U_1$,
$\tilde\mu_i \cdot \tilde\lambda_i = (-1)^i$, $i=1,\dots,4$.
%\[
%\tilde\mu_i \cdot \tilde\lambda_i =
%\begin{cases}
%-1 & i = 1,3 \\
%1 & i = 2,4.
%\end{cases}
%\]
Since the orientation of $S$ agrees with the boundary orientation of $\partial V_i$ exactly when $i=1$ or $3$, we see that $\tilde \mu_i \cdot \tilde \lambda_i = -1$ with respect to the boundary orientation of $V_i$ for all $i=1,\dots,4$. Thus, it follows that a curve of type $p_i \tilde \mu_i + q_i \tilde \lambda_i$ bounds a compressing disk for $V_i$.

Let $\gamma_1'$ and $\gamma_4'$ be parallel copies of $\gamma_1$ and $\gamma_4$, as shown in Figure \ref{f: genus2template}, and define
\begin{align*}
\alpha_1 &= p_1 \gamma_1' + q_1 \gamma_6, & \beta_1 &= p_2 \gamma_1 + q_2 \gamma_2, \\
\alpha_2 &= p_3 \gamma_3 - q_3 \gamma_4, & \beta_2 &= -p_4 \gamma_5 + q_4 \gamma_4'.
\end{align*}
Then $\alpha_1$ and $\alpha_2$ are disjoint compressing disks for $U_1$, and $\beta_1$ and $\beta_2$ are disjoint compressing disks for $U_2$, so $(S, \alpha_1 \cup \alpha_2, \beta_1 \cup \beta_2)$ presents $\Sigma(L)$. By construction,
\[
(S, \alpha_1\cup \alpha_2, \beta_1\cup\beta_2) = H(a_1,a_2,a_3,a_4),
\]
as required.
\end{proof}

%By construction, $U_1 = V_1 \bconn V_3$ and $U_2 = V_2 \bconn V_4$, where $\bconn$ denotes the boundary connected sum.

%\bibliographystyle{myalpha}
%\bibliography{/Users/Josh/Desktop/Papers/References}
%
%\end{document} 

\subsection{Strong, 1-extendible Heegaard diagrams of genus 2}

We will now show that any strong L-space admitting a genus-2 strong Heegaard diagram can be represented by a Heegaard diagram of the form in Construction \ref{const: H(a1a2a3a4)}. The main technical result of this section is:

\begin{proposition} \label{p: delta curves}
Suppose $H$ is a 1-extendible, irreducible, strong Heegaard diagram of genus 2 such that $\abs{\alpha_i \cap \beta_j} \ge 2$ for each $i,j$. Then there exist simple closed curves $\delta_1$ and $\delta_2$ on $S$ such that:
\begin{itemize}
\item Each $\delta_i$ divides $S$ into a union $S = T_i \cup U_i$ of two genus-$1$ surfaces;
\item $\delta_1$ separates $\alpha_1$ from $\alpha_2$, and $\delta_2$ separates $\beta_1$ from $\beta_2$;
\item $\abs{\delta_1 \cap \delta_2} =4$;
\item $T_1 \cap T_2$ and $U_1 \cap U_2$ are quadrilaterals whose boundaries each consist of two arcs of $\delta_1$ and two arcs of $\delta_2$, and the intersection of either of these regions with any $\alpha$ or $\beta$ circle consists of arcs that connect opposite sides of the quadrilateral;
\item $T_1 \cap U_2$ and $U_1 \cap T_2$ are annuli whose boundary components each consist of an arc of $\delta_1$ and an arc of $\delta_2$, and the intersection of either of these regions with any $\alpha$ or $\beta$ circle consists of arcs that connect opposite boundary components.
\end{itemize}
\end{proposition}

Before proving Proposition \ref{p: delta curves}, we show how it implies Theorem \ref{t: genus2}.

\begin{proof}[Proof of Theorem \ref{t: genus2}]
Let $Y$ be a strong L-space with a genus-2 strong Heegaard diagram $H$.  By Proposition \ref{p: 1-extendible}, we may assume that $H$ is 1-extendible.  If $H$ is reducible, then $Y$ is a connected sum of lens spaces. Furthermore, if any $\alpha$ circle meets a $\beta$ circle in a single point, then we may destabilize to obtain a genus-$1$ diagram, so $Y$ is a lens space. Thus, we may assume that $\abs{\alpha_i \cap \beta_j} \ge 2$ for each $i,j$.

Form the decompositions $S = T_1 \cup U_1 = T_2 \cup U_2$ as in Proposition \ref{p: delta curves}, and assume that the curves are labeled such that $\alpha_1 \subset T_1$, $\alpha_2 \subset U_1$, $\beta_1 \subset U_2$, and $\beta_2 \subset T_2$. The intersection of $\alpha_1 \cup \beta_1$ with the annulus $A_1 = T_1 \cap U_2$ satisfies the hypotheses of Lemma \ref{l: annulus}, so this intersection can be identified with a standard configuration of the form $(A,\bm\gamma(p,q,r,s))$, as described in Construction \ref{const: annulus}, for some integers $p,q,r,s$. By attaching the quadrilaterals $T_1 \cap T_2$ and $U_1 \cap U_2$, we see that $T_1 \cup U_2$ is a linear plumbing of three annuli, denoted $N$, and we can then identify $(T_1 \cup U_2, \alpha_1, \beta_1)$ with the configuration $(N, \eta_{p/q}, \eta_{r/s})$ from Construction \ref{const: annulus}. In a similar manner, $(T_2 \cup U_1, \alpha_2, \beta_2)$ can be identified with $(N', \eta_{p'/q'}, \eta_{r'/s'})$, where $N'$ is homeomorphic to $N$. This identifies $H$ with the description of $H(-\frac{q}{p}, \frac{s}{r}, \frac{p'}{q'}, -\frac{r'}{s'})$ as given in Construction \ref{const: H(a1a2a3a4)}, so $Y \cong \Sigma(L(-\frac{q}{p}, \frac{s}{r}, \frac{p'}{q'}, -\frac{r'}{s'}))$. Moreover, since $H$ is a strong diagram, we deduce that $-\frac{q}{p}$, $\frac{s}{r}$, $\frac{p'}{q'}$, and $-\frac{r'}{s'}$ all have the same sign, and therefore $L(-\frac{q}{p}, \frac{s}{r}, \frac{p'}{q'}, -\frac{r'}{s'})$ is an alternating link.
\end{proof}

\begin{proof}[Proof of Proposition \ref{p: delta curves}]
To begin, assume that the curves are labeled such that $\abs{\alpha_1 \cap \beta_1}$ is minimal among the four pairs of curves. Choose orientations so that $\alpha_1$ intersects both $\beta_1$ and $\beta_2$ positively, and $\alpha_2$ intersects $\beta_1$ positively and $\beta_2$ negatively. Let $v = \abs{\alpha_1 \cap \beta_1}$.

Let $R_1, \dots, R_n$ denote the regions of $S - \alpha_1 - \beta_1$. Let $e$ denote the Euler measure (see, e.g., \cite[Lemma 6.2]{OSz4Manifold}).  Then $-2 = e(S) = \sum e(R_i)$.  Since all intersection points in $\alpha_1 \cap \beta_1$ have the same sign, as we traverse a boundary component of any $R_i$, the arcs of $\alpha_1$ (resp.~$\beta_1$) must alternate in orientation. In particular, the number of arcs in each boundary component of each $R_i$ is a multiple of 4. It follows that $e(R_i) \leq 0$ for all $i$.  Thus, each $R_i$ may be a disk with $4$, $8$, or $12$ sides ($e = 0,-1,-2$, respectively); an annulus with two $4$-sided boundary components ($e = -2$); or a genus-$1$ surface with one $4$-sided boundary component ($e=-2$).  Consequently, $\{ R_1, \dots, R_n \}$ consists of a number of $4$-sided disks (which we will refer to as rectangles), along with either \begin{inparaenum} \item \label{case:two8gons} two $8$-sided disks; \item \label{case:12gon} a 12-sided disk; \item \label{case:annulus} an annulus; or \item \label{case:torus} a once-punctured torus. \end{inparaenum}

Let $N$ denote a regular neighborhood of $\alpha_1 \cup \beta_1$; note that $\chi(N) = -v$. Note that $-2 = \chi(S) = \chi(N) + \sum_{i=1}^n \chi(R_i)$, so $\sum_{i=1}^n \chi(R_i) = v-2$. Thus, the number of rectangles equals $v-4$ in case \ref{case:two8gons}, $v-3$ in case \ref{case:12gon}, $v-2$ in case \ref{case:annulus}, and $v-1$ in case \ref{case:torus}. We may assume that the non-rectangular region(s) are labeled $R_1$ (and $R_2$, in case \ref{case:two8gons}).

If $R_i$ is a rectangle, then any segment of $\alpha_2$ or $\beta_2$ that meets $R_i$ must enter and exit $R_i$ on opposite sides. Furthermore, $R_i$ cannot meet both $\alpha_2$ and $\beta_2$, since otherwise there would be a rectangle whose four sides are arcs of $\alpha_1$, $\beta_1$, $\alpha_2$, and $\beta_2$, which is prohibited in a strong diagram. Thus, all of the points in $\alpha_2 \cap \beta_2$ (which is a non-empty intersection) must lie in the non-rectangular region(s) of $S - \alpha_1 - \beta_2$. We consider the four cases enumerated above.

\begin{description}[leftmargin=0in, listparindent=\parindent, itemsep=3pt]
\item[Case \ref{case:two8gons} (two octagons)]
Without loss of generality, suppose that $\alpha_2$ and $\beta_2$ intersect inside the octagonal region $R_1$. Label the edges of $R_1$ consecutively by $e_0, \dots, e_7$, so that $e_0, e_2, e_4, e_6$ are arcs of $\alpha_1$ and $e_1, e_3, e_5, e_7$ are arcs of $\beta_1$. Let $a$ and $b$ be the components of $\alpha_2 \cap R_i$ and $\beta_2 \cap R_1$, respectively, containing some point of $\alpha_2 \cap \beta_2$, and assume that $b$ intersects $e_0$. Since all points of $\alpha_2 \cap \beta_0$ have the same sign, the other endpoint of $b$ must be on either $e_2$ or $e_6$. However, then $a$ must intersect either $e_1$ or $e_7$, respectively, and we obtain a rectangle whose four sides are arcs of $\alpha_1$, $\beta_1$, $\alpha_2$, and $\beta_2$, which is prohibited. Thus, Case \ref{case:two8gons} cannot occur.

\item[Case \ref{case:12gon} (one dodecagon)]

\begin{figure}
\labellist
 \pinlabel {{\color{red} $\alpha_1$}} at 181 249
 \pinlabel {{\color{blue} $\beta_1$}} at 151 258
 \pinlabel {{\color{darkred} $\alpha_2$}} at 196 138
 \pinlabel {{\color{darkblue} $\beta_2$}} at 116 197
 \small
 \pinlabel {{\color{red} $e_0$}} at 144 241
 \pinlabel {{\color{red} $e_2$}} at 229 187
 \pinlabel {{\color{red} $e_4$}} at 223 66
 \pinlabel {{\color{red} $e_6$}} at 114 19
 \pinlabel {{\color{red} $e_8$}} at 31 74
 \pinlabel {{\color{red} $e_{10}$}} at 35 194
 \pinlabel {{\color{blue} $e_1$}} at 194 221
 \pinlabel {{\color{blue} $e_3$}} at 242 141
 \pinlabel {{\color{blue} $e_5$}} at 187 31
 \pinlabel {{\color{blue} $e_7$}} at 64 39
 \pinlabel {{\color{blue} $e_9$}} at 19 118
 \pinlabel {{\color{blue} $e_{11}$}} at 73 230
\endlabellist
\includegraphics{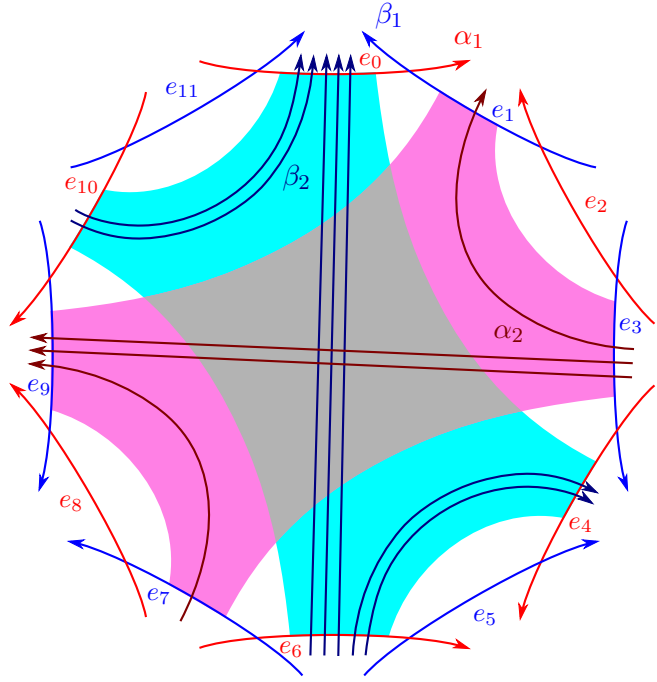}
\caption{A 12-gon region of $S \minus (\alpha_1 \cup \beta_1)$. The octagons $D_1,D_2$ found in Case \ref{case:12gon} of the proof of Proposition \ref{p: delta curves} are shown in pink and light blue, respectively, overlapping in the gray region.} \label{f: 12gon}
\end{figure}

Suppose that $\alpha_2$ and $\beta_2$ intersect inside the dodecagon $R_1$. Label the edges of $R_1$ clockwise by $e_0, \dots, e_{11}$, so that odd indices correspond to arcs of $\alpha_1$ and even indices correspond to arcs of $\alpha_2$. We consider these indices as elements of $\Z/12$.

The edges $e_i$ inherit orientations from the orientations on $\alpha_1$ and $\beta_1$. Since $\alpha_1$ and $\beta_1$ intersect positively, we may assume that $e_0, e_3, e_4, e_7, e_8, e_{11}$ are oriented clockwise (i.e. opposite to the boundary orientation on $\partial R_1$), and the remaining edges are oriented counterclockwise. By our orientation conventions, segments of $\alpha_2$ may enter $R_1$ through $e_3$, $e_7$, and $e_{11}$ and exit through $e_1$, $e_5$, and $e_9$; segments of $\beta_2$ may enter through $e_2$, $e_6$, and $e_{10}$ and exit through $e_0$, $e_4$, and $e_8$.

\begin{lemma} \label{l: 12gon identifications}
For $i \in \Z/12$, if $\gamma_i$ is an arc that begins in $R_1$, exits through $e_i$, and proceeds, passing through rectangular regions of $S - \alpha_1 - \beta_1$ without turning left or right, until it reenters $R_1$ through some $e_j$, then $j=i+6$. (That is, $\gamma_i$ reenters $R_1$ through the edge directly opposite the one through which it exited.)
\end{lemma}

\begin{proof}
For notational simplicity, we consider the case where $i=0$; the others proceed similarly.

Because the arc $\gamma_i$ runs parallel to $\beta_1$ (up to isotopy), it intersects $\alpha_1$ positively when it reenters $R_1$ through $e_j$. This implies that $j=2$, $6$, or $10$.

Let $\beta_1'$ and $\beta_1''$ be parallel copies of $\beta_1$ that run just to the left and right of $\beta_1$, respectively. Note that $\beta_1' \cap R_1$ consists of arcs $e_1'$, $e_5'$, and $e_9'$ that run parallel to $e_1$, $e_5$, and $e_9$, respectively, where $e_k'$ runs from a point in the interior of $e_{k+1}$ to a point in the interior of $e_{k-1}$. Up to isotopy, we may take $\gamma_0$ to be a segment of $\beta_1'$ that begins on $e_1'$ and ends on $e_1'$, $e_5'$, or $e_9'$, depending on whether $j=2$, $6$, or $10$, respectively. If $j=2$, then $\gamma_0 \cup e_1'$ is an embedded circle that runs parallel to $\beta_1$ on the left, meaning that it is actually equal to all of $\beta_1'$, a contradiction. A similar argument using $\beta_1''$ excludes the $j=10$ case. Thus, we conclude that $j=6$.
\end{proof}

%For $i \in \Z/12$, let $\Gamma_i$ be the union of all the rectangular regions of $S - \alpha_1 - \beta_1$ through which $\gamma_i$ passes, and let $p_i = \abs{ \gamma_i \cap (\alpha_1 \cup \beta_1)}$. (Note that $\gamma_i$ intersects either $\alpha_1$ or $\beta_1$, but not both.) Clearly, $p_i = p_{i+6}$, and $\Gamma_i$ is the union of $p_i-1$ rectangles. Since every rectangle is contained in a unique $\Gamma_i$ with $i \in \{0,2,4\}$ and a unique $\Gamma_j$ with $j \in \{1,3,5\}$, we have
%\begin{equation}
%p_0 + p_2 + p_4 = p_1 + p_3 + p_5 = v-3.
%\end{equation}

For $i \in \Z/12$, let $p_i = \abs{ \gamma_i \cap (\alpha_1 \cup \beta_1)}$, so that $\gamma_i$ passes through $p_i-1$ rectangular regions. Clearly, $p_i = p_{i+6}$. Since every rectangle meets a unique $\gamma_i$ with $i \in \{0,2,4\}$ and a unique $\gamma_j$ with $j \in \{1,3,5\}$, we have
\begin{equation}
p_0 + p_2 + p_4 = p_1 + p_3 + p_5 = v.
\end{equation}

Let $a_0$ and $b_0$ be the components of $\alpha_2 \cap R_i$ and $\beta_2 \cap R_1$, respectively, containing some point of $\alpha_2 \cap \beta_2$. Up to relabeling, we may assume that $b_0$ exits $R_1$ through $e_0$.

\begin{lemma} \label{l: 12gon crossings}
If $a$ and $b$ are any components of $\alpha_2 \cap R_i$ and $\beta_2 \cap R_1$, respectively, that intersect in $R_1$, then $a$ enters $R_1$ through $e_3$ and exits through $e_9$, and $b$ enters through $e_6$ and exits through $e_0$.
\end{lemma}

\begin{proof}
Suppose $b$ enters $R_1$ through $e_i$, where $i \in \{2,6,10\}$. Then $a$ cannot intersect $e_{i \pm 1}$, since that would produce a prohibited rectangle. Therefore, $b$ must exit $R_1$ through $e_{i+6}$, or else $a$ would be forced to intersect $e_{i \pm 1}$. Symmetrically, $a$ must enter $R_1$ through $r_j$, where $j \in \{3,7,11\}$, and exit through $r_{j+6}$. It follows that $j=i+3$. Furthermore, if $a'$ and $b'$ are another pair of components that intersect, then $a'$ and $b'$ must enter and exit through the same sides as $a$ and $b$, respectively, since otherwise $a$ and $a'$ (resp.~$b$ and $b'$) would intersect, violating the embeddedness of $\alpha_2$ (resp.~$\beta_2$). Since we know that $b_0$ exits through $e_0$, we conclude that the same is true for any $b$.
\end{proof}

It follows from the previous two lemmas that for some nonnegative integers $x,y,z,x',y',z'$, with $y,y' \ge 1$, we have:
\begin{itemize}
\item $\alpha_2 \cap R_1$ consists of $x$ arcs from $e_3$ to $e_5$, $x$ arcs from $e_{11}$ to $e_9$, $y$ arcs from $e_3$ to $e_9$, $z$ arcs from $e_3$ to $e_1$, and $z$ arcs from $e_7$ to $e_9$.

\item $\beta_2 \cap R_1$ consists of $x'$ arcs from $e_6$ to $e_8$, $x'$ arcs from $e_2$ to $e_0$, $y'$ arcs from $e_6$ to $e_0$, $z'$ arcs from $e_6$ to $e_4$, and $z'$ arcs from $e_{10}$ to $e_0$.
\end{itemize}

Therefore, we have
\begin{align*}
\abs{\alpha_2 \cap \beta_1} &= zp_1 + (x+y+z) p_3 + xp_5 \\
\abs{\alpha_1 \cap \beta_2} &= (x'+y'+z') p_0 + x' p_2 + z' p_4 \\
\abs{\alpha_2 \cap \beta_2} &= yy'.
\end{align*}
We make the following three observations:
\begin{enumerate}
\item
At least one of $x,z$ is zero, since otherwise every rectangular region of $S \minus (\alpha_1 \cup \alpha_2)$ would meet $\alpha_2$, and therefore $\beta_2$ would be constrained to lie in $R_1$, a contradiction. Similarly, at least one of $x',z'$ is zero.

\item
At least one of $x, z'$ is zero, since otherwise the arcs of $\alpha_2$ from $e_3$ to $e_5$ would intersect the arcs of $\beta_2$ from $e_6$ to $e_4$, violating Lemma \ref{l: 12gon crossings}. Similarly, at least one of $x',z$ is zero.

\item
If $x=z=0$, then $\alpha_2$ consists entirely of $y$ parallel copies of a circle that runs from $e_6$ to $e_0$ in $R_1$ and then from $e_0$ back to $e_6$ outside of $R_1$. Since $\alpha_2$ is a single circle, we then have $y=1$. By the minimality of $\abs{\alpha_1 \cap \beta_1}$,
\[
p_3 = \abs{\alpha_2 \cap \beta_1} \ge \abs{\alpha_1 \cap \beta_1} = p_1 + p_3 + p_5,
\]
a contradiction. Thus, $x$ and $z$ are not both zero. Similarly, $x'$ and $z'$ are not both zero.
\end{enumerate}

Without loss of generality, let us assume that $x=0, z \ne 0, x'=0, z' \ne 0$; the opposite case is handled symmetrically. We can choose a pair of embedded octagons $D_1, D_2 \subset R_1$ such that:
\begin{itemize}
\item $D_1$ contains $\alpha_2 \cap R_1$ and edges on $e_1$, $e_3$, $e_7$, and $e_9$;

\item $D_2$ contains $\beta_2 \cap R_1$ has edges on $e_0$, $e_4$, $e_6$, and $e_{10}$;

\item $D_1 \cap D_2$ is a quadrilateral that contains $\alpha_2 \cap \beta_2$.
\end{itemize}
(See Figure \ref{f: 12gon}.) We now define $T_1$ (resp.~$T_2$) to be the union of $D_1$ (resp.~$D_2$) with strips that are tubular neighborhoods of $\gamma_1$ and $\gamma_3$ (resp.~$\gamma_0$ and $\gamma_4$); this is a genus-$1$ surface with one boundary component that contains $\alpha_2$ (resp.~$\beta_2$) in its interior but avoids $\alpha_1$ (resp.~$\beta_1$). The strips attached to $D_1$ cannot intersect the strips attached to $D_2$, since they contain segments of $\alpha_2$ and $\beta_2$, respectively; thus, $T_1 \cap T_2 = D_1 \cap D_2$, as required. Defining $U_1 = \overline{S \minus T_1}$ and $U_2 = \overline{S \minus T_2}$, it is easy to see that $T_1 \cap U_2$ and $T_2 \cap U_1$ are annuli and that $U_1 \cap U_2$ is a quadrilateral, as required.

\item[Case \ref{case:annulus} (one annulus)]

\begin{figure}
\labellist
 \pinlabel {{\color{red} $\alpha_1$}} at 318 301
 \pinlabel {{\color{blue} $\beta_1$}} at 286 332
 \pinlabel {{\color{darkred} $\alpha_2$}} at 6 170
 \pinlabel {{\color{darkblue} $\beta_2$}} at 162 332
 \small
 \pinlabel {{\color{red} $e_0$}} at 210 319
 \pinlabel {{\color{red} $e_2$}} at 162 21
 \pinlabel {{\color{red} $e_4$}} at 180 216
 \pinlabel {{\color{red} $e_6$}} at 162 122
 \pinlabel {{\color{blue} $e_1$}} at 20 237
 \pinlabel {{\color{blue} $e_3$}} at 305 170
 \pinlabel {{\color{blue} $e_5$}} at 208 188
 \pinlabel {{\color{blue} $e_7$}} at 117 170
\endlabellist
\includegraphics[scale=0.9]{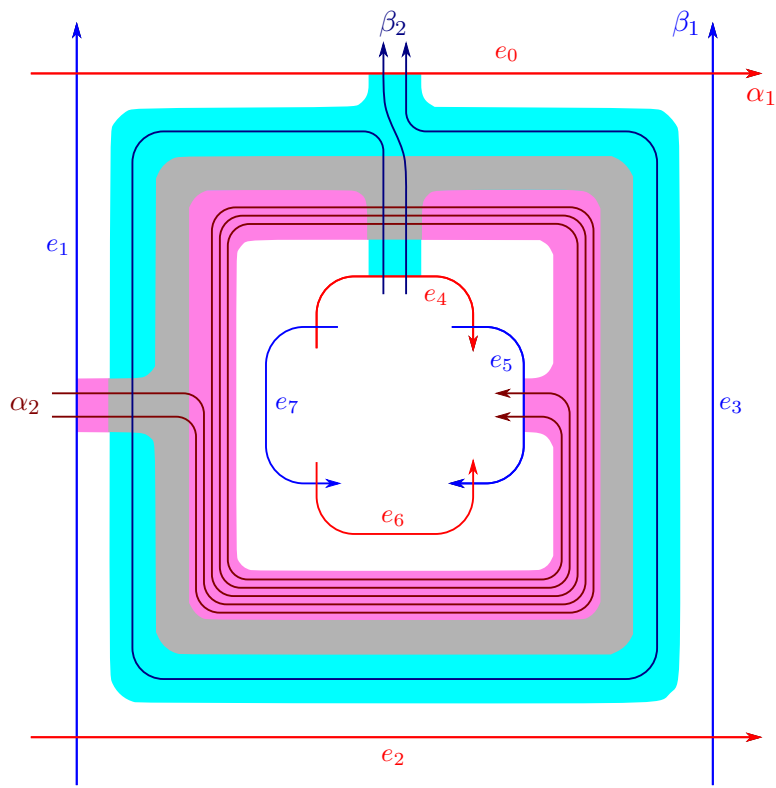}
\caption{An annulus region of $S \minus (\alpha_1 \cup \beta_1)$. The annuli $N_1,N_2$ found in Case \ref{case:annulus} of the proof of Proposition \ref{p: delta curves} are shown in pink and light blue, respectively, overlapping in the gray region.} \label{f: annulus}
\end{figure}

Label the edges of $R_1$ $e_0, \dots, e_7$, where:
\begin{itemize}
\item $e_0$, $e_2$, $e_4$, and $e_6$ are segments of $\alpha_1$, and $e_1$, $e_3$, $e_5$, and $e_7$ are segments of $\beta_1$;

\item The edges on one boundary component of $R_1$ are $e_0, e_1, e_2, e_3$, ordered according to the boundary orientation. The edges on the other component are $e_4, e_5, e_6, e_7$, ordered likewise.

\item The orientations on $e_2$, $e_3$, $e_4$, and $e_5$ agree with the boundary orientation, while the orientations on $e_0$, $e_1$, $e_6$, and $e_7$ disagree with the boundary orientation.
\end{itemize}

An argument just like Lemma \ref{l: 12gon identifications} shows:

\begin{lemma} \label{l: annulus identifications}
For $i \in \Z/8$, if $\gamma_i$ is an arc that begins in $R_1$, exits through $e_i$, and proceeds, passing through rectangular regions of $S - \alpha_1 - \beta_1$ without turning left or right, until it reenters $R_1$ through some $e_j$, then $j=i+4$. \qed
\end{lemma}

%As before, for $i \in \Z/8$, let $p_i = \abs{\gamma_i \cap (\alpha_1 \cup \beta_2)}$, so that clearly $p_i = p_{i+4}$, and let $\Gamma_i$ be the union of the $p_i-1$ rectangles through which $\gamma_i$ passes. We have
%\[
%p_0 + p_2 = p_1 + p_3 = v-2.
%\]
%As before, let $\Gamma_i$ be the union of the $p_i-1$ rectangles through which $\gamma_i$ passes.

According to our orientation conventions, segments of $\alpha_2$ may enter $R_1$ through $e_1$ or $e_7$ and exit through $e_3$ or $e_5$, and segments of $\beta_2$ may enter through $e_2$ or $e_4$ and exit through $e_0$ and $e_6$.
\begin{lemma} \label{l: annulus allarcssame}
Every component of $\alpha_2 \cap R_1$ enters $R_1$ through the same edge (say $e_i$, where $i \in \{1,7\}$) and exits through $e_{i+4}$. Likewise, every component of $\beta_2 \cap R_1$ enters $R_1$ through the same edge (say $e_j$, where $j  \in \{2,4\}$) and exits through $e_{j+4}$.
\end{lemma}

\begin{proof}
If there is some component $a$ of $\alpha_2 \cap R_1$ that intersects $e_i$, and some component $a'$ (possibly the same as $a$) that intersects $e_{i\pm 2}$, then by Lemma \ref{l: annulus identifications}, $\alpha_2$ meets every rectangular region of $S - \alpha_1 - \beta_1$, which implies that $\alpha_2$ does not meet any rectangular region. However, this implies that $\beta_2 \cap \alpha_1 = \emptyset$, a contradiction. An analogous argument applies for $\beta_2$.
\end{proof}

Assume for concreteness that segments of $\alpha_2$ enter $R_1$ through $e_1$ and exit through $e_5$, and segments of $\beta_2$ enter $R_1$ through $e_4$ and exit through $e_0$. (The other cases are handled symmetrically.) Lemma \ref{l: annulus allarcssame} together implies that the multicurves $\alpha_2 \cap R_1$ and $-\beta_2 \cap R_1$ satisfy the hypotheses of Lemma \ref{l: annulus}; therefore, $(R_1, R_1 \cap (\alpha_2 \cup \beta_2))$ may be identified with the standard configuration given in Construction \ref{const: annulus}. As shown in Figure \ref{f: annulus}, there are embedded annuli $N_1, N_2 \subset R_1$ such that:
\begin{itemize}
\item $N_1$ contains $\alpha_2 \cap R_1$, and $\partial N_1 \cap \partial R_1$ consists of proper subarcs of $e_1$ and $e_5$;
\item $N_2$ contains $\beta_2 \cap R_1$, and $\partial N_2 \cap \partial R_1$ consists of proper subarcs of $e_0$ and $e_4$;
\item $N_1 \cap N_2$ is an annulus, and $\abs{\partial N_1 \cap \partial N_2} = 4$.
\end{itemize}
Let $T_1$ (resp.~$U_2$) be the union of $N_1$ (resp.~$N_2$) with a $1$-handle that follows $\gamma_1$ (resp.~$\gamma_0$); this is then a genus-$1$ surface with one boundary component that contains $\alpha_2$ (resp.~$\beta_2$) in its interior but avoids $\alpha_1$ (resp.~$\beta_1$). Moreover, $T_1 \cap U_2$ is an annulus whose boundary components each consist of an arc of $\delta_1 = \partial T_1$ and an arc of $\delta_2 = \partial U_2$. If we define $U_1 = \overline{S \minus T_1}$ and $T_2 = \overline{S \minus U_2}$, it is easy to verify that the remaining intersections are as required.

\item[Case \ref{case:torus} (one punctured torus)]

Label the edges of $R_1$ $e_0, e_1, e_2, e_3$, so that $e_0$ and $e_2$ are segments of $\alpha_1$, and $e_1$ and $e_3$ are segments of $\beta_1$. In this case, a path $\gamma_i$ that exits $R_1$ through $e_i$ will proceed through all $v-1$ rectangular regions before reentering through $e_{i+2}$ (indices modulo $4$). Thus, we see that $\alpha_2$ and $\beta_2$ cannot both exit $R_1$, or else they would be forced to intersect in a rectangular region of $S - \alpha_1 - \alpha_2$. Hence, either $\alpha_1 \cap \beta_2=\emptyset$ or $\alpha_2 \cap \beta_1 = \emptyset$, violating our assumption that $\abs{\alpha_i \cap \beta_j} \ge 2$. Thus, case \ref{case:torus} does not arise.

\end{description}

This concludes the proof of Proposition \ref{p: delta curves}.
\end{proof}

\begin{corollary}
\label{c: mingenus}
The minimum genus of a strong Heegaard diagram for a strong L-space can exceed its Heegaard genus.
\end{corollary}

This corollary is proven by either of the following examples:

\begin{example} \label{ex: borromean}
Let $L$ denote the Borromean rings, which is a prime, three-component, alternating link of bridge number 3. Therefore, $Y=\Sigma(L)$ is an irreducible strong L-space of Heegaard genus two. Suppose $Y$ can represented by a strong Heegaard diagram $H$ of genus $2$, which may be assumed to be $1$-extendible by Proposition \ref{p: 1-extendible}. Since $Y$ is irreducible, $M(H)$ does not contain a $0$ entry.  Since $M(H)$ is a $2 \times 2$ presentation matrix for $H_1(Y;\bZ) \cong \bZ/4 \oplus \bZ/4$, every entry in $M(H)$ is a non-zero multiple of 4. However, since $M(H)$ is a P\'olya matrix, we must then have
\[
\abs{\det(M(H))} \ge 4 \cdot 4 + 4 \cdot 4 > 16,
\]
a contradiction. Thus, any strong Heegaard diagram for $Y$ has genus greater than $2$.

More generally, it is straightforward to find prime, alternating 3-bridge diagrams without non-trivial Conway circles.  By the solution of the Tait flyping conjecture in this case \cite{mt:tait, schrijver:tait, GreeneLattices}, no such link is isotopic to a link of the form $L(a_1,a_2,a_3,a_4)$ with all $a_i$ the same sign.  Therefore, its branched double cover is a strong L-space of Heegaard genus $2$, but it does not possess a strong Heegaard diagram of genus $2$.
\end{example}

\begin{example} \label{ex: dunfield}
Using the program SnapPy, Nathan Dunfield has identified $316$ strong L-spaces among the $11,031$ small hyperbolic manifolds in the Hodgson--Weeks census. Most of these manifolds have Heegaard genus $2$, and all have Heegaard genus at most $3$. Theorem \ref{t: genus2} implies that these manifolds cannot admit strong diagrams of genus $2$. In fact, Dunfield was only able to find strong Heegaard diagrams of genus $4$ through $7$ for any of these manifolds. Moreover, all of Dunfield's examples can be realized as Dehn fillings of the minimally twisted $5$--chain link, so they are branched double covers of links obtained from filling rational tangles into the pentacle graph, and all of these links are in fact alternating. (See \cite{DunfieldThurstonExperiements, MartelliPetronioRoukema}.)
\end{example}

Corollary \ref{c: mingenus} motivates the following questions:

\begin{question} \label{q: stronggenus}
Can the difference between the Heegaard genus of a strong L-space $Y$ and the minimal genus of a strong Heegaard diagram for $Y$ be arbitrarily large?
\end{question}

\begin{question} \label{q: graph}
Is every strong L-space admitting a strong Heegaard diagram of genus $3$ a Seifert fibered space or graph manifold?
\end{question}

\section{Floer simple knots} \label{sec: simpleknots}

Given a Heegaard diagram $H = (S,\alpha,\beta)$ for a $3$-manifold $Y$, a pair of basepoints $w,z \in S \minus (\alpha \cup \beta)$ determines an oriented knot $K \subset Y$, as explained in \cite{OSzKnot}. Briefly, we let $\gamma_\alpha$ (resp.~$\gamma_\beta$) be an arc from $w$ to $z$ in the complement of the $\alpha$ (resp.~$\beta$) circles, pushed into the $\alpha$ (resp.~$\beta$) handlebody, and we set $K = \gamma_\alpha \cup -\gamma_\beta$. The pair of basepoints induce a filtration on the complex $\CF(H)$ whose associated graded complex is denoted $\CFK(H)$, and the homology of $\CFK(H)$ is an invariant called the \emph{knot Floer homology} of $K$, denoted $\HFK(Y,K)$ \cite{OSzKnot, RasmussenThesis}.

A knot $K$ in an L-space $Y$ is called \emph{Floer simple} if $\HFK(Y,K) \cong \HF(Y)$. The classification of Floer simple knots is of great interest, in part because any Floer simple knot minimizes rational genus in its homology class \cite{NiWuRational}. In particular, the only nulhomologous knot in $Y$ that is Floer simple is the unknot. Note that there exist homology classes in certain L-spaces that do not contain any Floer simple knot \cite[Proposition 7.4]{LevineRubermanStrleNonorientable}.

Strong Heegaard diagrams provide a natural source of Floer simple knots. Given a strong Heegaard diagram $H$ for a strong L-space $Y$, a knot $K$ is called \emph{simple with respect to $H$} if it can be represented by a pair of basepoints in $H$; it follows that $K$ is Floer simple. Each homology class in a lens space contains a unique simple knot (with respect to the genus-1 strong Heegaard diagram). One formulation of the famous Berge conjecture asserts that any knot in a lens space with an integer-slope surgery to $S^3$ must be a simple knot \cite{berge:lens}. This would follow from the conjecture that every Floer simple knot in a lens space is simple \cite{HeddenBerge, RasmussenLens}.

Thus, topological considerations motivate the investigation of the existence and uniqueness of simple knots in strong L-spaces, as in the following proposition.

\begin{proposition}\label{p: simple}
If $Y$ is presented by a strong diagram $H$ of genus 2, then every homology class in $H_1(Y;\bZ)$ is represented by a simple knot with respect to $H$. \qed
\end{proposition}

Proposition \ref{p: simple} follows at once from the following theorem.

%There are several natural variations on the question of existence of (Floer) simple knots in an L-space $Y$:
%\begin{enumerate}
%\item
%Does $Y$ contain any (non-trivial) Floer simple knots?
%\item
%Which homology classes support Floer simple knots?
%\item
%Do there exist homologous, non-isotopic Floer simple knots?
%\item
%What are all Floer simple knots in $Y$?
%\end{enumerate}

\begin{theorem}
\label{t: basepoints}%(Trieste; August 10, 2012)
Given a genus two Heegaard diagram $H = (S,\alpha,\beta)$ of a rational homology sphere $Y$ and any homology class $x \in H_1(Y;\bZ)$, there exists a pair of basepoints $z_1,z_2 \in S$ such that the doubly-pointed Heegaard diagram $(S,\alpha,\beta,z_1,z_2)$ represents a knot $K \subset Y$ with $[K] = x$.
\end{theorem}

Theorem \ref{t: basepoints} does not generalize to diagrams of higher genus.  As an example, take a genus three Heegaard diagram $H$ for $Y = \#^3 L(p,q)$ by forming the connected sum of three genus one strong Heegaard diagrams for its individual summands.  Then $H$ contains $3p-2$ regions, so the knots in $Y$ represented by pairs of basepoints in $H$ represent at most $(3p-2)^2$ homology classes, whereas $H_1(Y;\bZ)$ has order $p^3$.

%\begin{proposition}\label{p: unique} (Miami to New York; December 2)
%\
%Given a genus two strong Heegaard diagram, homologous simple knots with respect to it are isotopic.
%\end{proposition}

\begin{proof}[Proof of Theorem \ref{t: basepoints}]
First we recall how to determine the homology class represented by a knot $K \subset Y$ presented by a doubly-pointed Heegaard diagram $H=(S,\alpha,\beta,z_1,z_2)$.  The $i^{\mathrm{th}}$ column of the intersection matrix $M(H')$ is the vector of intersection numbers
\[
v_i = (\alpha_1 \cdot \beta_i, \dots, \alpha_g \cdot \beta_i) \in \bZ^g,
\]
and we have an isomorphism
\[
H_1(Y;\bZ) \cong \coker(M(H')) \cong \bZ^g / \gen{v_1,\dots,v_g}.
\]
Choose an arc $\gamma \subset S \minus \beta$ transverse to the $\alpha$ curves that begins at $z_1$ and ends at $z_2$, and form the vector $v(\gamma) = (\alpha_1 \cdot \gamma, \dots, \alpha_g \cdot \gamma) \in \bZ^g$.  Then under the above isomorphism, the class $[K]$ maps to the image of $v$ in $\bZ^g / \gen{v_1,\dots,v_g}$.

Now we return to the case at hand.  Choose a region of $H$ whose boundary contains arcs of both $\beta_1$ and $\beta_2$, and place a basepoint $z_0$ in it.  For $i=1,2$, traverse a parallel copy of $\beta_i'$ of $\beta_i$ based at $z_0$ and place a basepoint on $\beta_i'$ in each region it enters.  To each subarc $\gamma \subset \beta_i'$ that starts at $z_0$ and ends at one of these basepoints, associate the vector $v(\gamma)= (\alpha_1 \cdot \gamma, \alpha_2 \cdot \gamma)$.  The collection of these vectors gives a sequence $S_i$ beginning with $(0,0)$ and ending with $v_i = (\alpha_1 \cdot \beta_i, \alpha_2 \cdot \beta_i)$, in which each pair of consecutive terms differs by $(\pm 1, 0)$ or $(0, \pm 1)$.

We claim that $S_1 - S_2 = \{s_1 - s_2 \mid s_i \in S_i\}$ contains a vector in each equivalence class of $\bZ^2 / \langle v_1, v_2 \rangle$.  Assuming this, choose $x \in \bZ^2 / \gen{v_1, v_2} \cong H_1(Y;\bZ)$ and represent $-x$ by  $s_1 - s_2 \in S_1 - S_2$.  Choose $\gamma_i \subset \beta_i'$ with $v(\gamma_i) = s_i$ and write $\del v_i = z_i - z_0$, $i=1,2$.  Then the concatenation of $-\gamma_1$ and $\gamma_2$ is an arc $\gamma \subset S \minus \beta$ from $z_1$ to $z_2$ with $v(\gamma) = -s_1 + s_2$.  It follows that the knot represented by the doubly-pointed Heegaard diagram $(S,\alpha,\beta,z_1,z_2)$ represents the class $x$.  Since $x$ was arbitrary, the Theorem follows from the asserted claim.

Now we establish the claim.  Let $\Gamma_i \subset \bR^2$ denote the rectilinear curve from $(0,0)$ to $v_i$ that interpolates the sequence $S_i$.  Regard $\Gamma_i$ as the image of a map $F_i\co I \to \bR^2$.  It follows that $F_i$ is homotopic rel endpoints to the linear map from $(0,0)$ to $v_i$, and it descends to a based map $f_i\co S^1 \to \bR^2 / \gen{v_1, v_2}$.  Since $Y$ is a rational homology sphere, it follows that $v_1$ and $v_2$ are linearly independent, $\bR^2 / \gen{v_1, v_2} \cong T^2$, and $\{[f_1], [f_2]\}$ is a basis for $H_1(T^2)$.  It follows that the map $f\co S^1 \times S^1 \to T^2$, $f(x,y) = f_1(x)-f_2(y)$, induces an isomorphism on $H_1$.  In particular, $f$ has degree one, so it surjects.  On the other hand, the image of $f$ is the image of the set $\Gamma_1- \Gamma_2 = \{ \gamma_1 - \gamma_2 \mid \gamma_i \in \Gamma_i \}$ in $\bR^2 / \gen{v_1, v_2}$.  We conclude that $\Gamma_1 - \Gamma_2$ contains a vector in each equivalence class of $\bR^2/ \gen{v_1, v_2}$ and hence each class of $\bZ^2 / \gen{v_1, v_2}$.

To complete the argument, we show that $S_1 - S_2 = (\Gamma_1 - \Gamma_2) \cap \bZ^2$.  Thus, suppose that $\gamma_1 - \gamma_2 \in \bZ^2$, $\gamma_i \in \Gamma_i$.  Then $\gamma_i$ lies on the unit segment (vertical or horizontal) between two consecutive points in $S_i$.  If $\gamma_1$ lies on a horizontal segment, then its $y$-coordinate is integral.  Since $\gamma_1-\gamma_2 \in \bZ^2$, $\gamma_2$ has an integral $y$-coordinate, so it lies on a horizontal segment, too.  In this case, we take $s_i \in S_i$ to be the left endpoint of the horizontal segment containing $\gamma_i$, and $s_1 - s_2 = \gamma_1 - \gamma_2$.  Similarly, if the $\gamma_i$ both lie on vertical segments, then the top endpoints furnish points $s_i \in S_i$ such that $s_1 - s_2 = \gamma_1 - \gamma_2$.  In summary, we recover the claim asserting that $S_1 - S_2$ contains a vector in each equivalence class of $\bZ^2 / \gen{v_1, v_2}$, and the Theorem follows as described.
\end{proof}

\begin{remark}
With more effort, it appears possible to show that homologous simple knots with respect to a given genus-2 strong Heegaard diagram are in fact isotopic.  This is trivially the case with genus-1 strong diagrams.  We do not know whether there exist two strong Heegaard diagrams $H_1$, $H_2$ for a space, and knots $K_1$, $K_2$ that are simple with respect to $H_1$, $H_2$, respectively, such that $[K_1]=[K_2]$ but $K_1 \ne K_2$.
\end{remark}

Proposition \ref{p: simple} has an interesting application concerning non-orientable surfaces in $4$-manifolds. If $Y$ is a rational homology sphere with $H_2(Y;\Z/2) \ne 0$, then any class $x \in H_2(Y;\Z/2)$ can be represented by an closed, embedded, non-orientable surface of some genus.  The minimal genus of such a surface is a measure of the topological complexity of $Y$, first studied by Bredon and Wood \cite{BredonWood}. It is natural to ask whether the genus can be lowered by adding a dimension: i.e., whether the minimal genus of a surface in $Y$ representing $x$ is the same as the minimal genus of a surface in $Y \times I$ representing the corresponding homology class. For the strong L-spaces considered in Section \ref{sec: genus2}, the answer to this question is negative:

\begin{corollary} \label{c: nonorientable}
Let $Y$ be a strong L-space admitting a strong Heegaard diagram of genus $2$, and let $W$ be any smooth homology cobordism from $Y$ to itself (e.g. $W = Y \times I$). For any nonzero homology class $x \in H_2(W;\Z/2)$, the minimal genus of a smoothly embedded, closed, connected, non-orientable surface in $W$ representing $x$ is the same as the minimal genus of such a surface in $Y$ itself representing the corresponding element $\bar x \in H_2(Y;\Z/2)$. Furthermore, this minimal genus is equal to
\begin{equation} \label{eq: genusbound}
\max_{\spincs \in \Spin^c(Y)} 2 \left( d(Y,\spincs) - d(Y, \spincs + \PD(\beta(\bar x))) \right),
\end{equation}
where $\beta$ denotes the Bockstein homomorphism $H_2(Y;\Z/2) \xrightarrow{\beta} H_1(Y;\Z)$.
\end{corollary}

Here $\Spin^c(Y)$ is the set of spin$^c$ structures on $Y$
%, which is an affine set for $H^2(Y;\Z)$,
and $d(Y,\spincs)$ denotes the \emph{correction term} of the spin$^c$ structure $\spincs$, a rational number derived from the Heegaard Floer homology of $Y$ \cite{OSzAbsolute}. When $Y$ is an L-space, $d(Y,\spincs)$ equals the absolute rational grading of the summand $\HF(Y,\spincs) \subset \HF(Y)$. Additionally, the correction terms of a strong L-space can be readily computed (up to an overall shift) from a strong Heegaard diagram $H$ using a formula of Lee and Lipshitz \cite{LeeLipshitz} for the grading difference between two generators of $\CF(H)$.  This information is enough to determine the quantity \eqref{eq: genusbound}.

\begin{proof}[Proof of Corollary \ref{c: nonorientable}]
Work of Ni and Wu \cite{NiWuRational} and of Ruberman, Strle, and the second author \cite{LevineRubermanStrleNonorientable} shows that \eqref{eq: genusbound} is a lower bound on the minimal genus in both $Y$ and $W$ (for any rational homology sphere $Y$), and that both of these bounds are sharp when $Y$ is an L-space and $\beta(\bar x)$ is represented by a Floer simple knot.
\end{proof}

\begin{question} \label{q: simpleknot}
If $Y$ is a strong L-space, is every homology class of order 2 represented by a simple knot with respect to some strong Heegaard diagram for $Y$?
\end{question}

An affirmative answer would imply that the conclusion of Corollary \ref{c: nonorientable} applies to every strong L-space.

\section{Waves, anti-waves, and weak reducibility}
\label{sec: waves}

\subsection{Waves}

A {\em wave} in a Heegaard diagram $H$ is an arc $\gamma$ that is properly embedded in a region $R$ of $H$ and whose endpoints lie on the interiors of distinct arcs of $\del R$ and on the same $\alpha$ or $\beta$ curve, such that the local signs of intersection at the two endpoints of $\gamma$ are opposite. Waves were introduced by Volodin, Kuznetsov, and Fomenko \cite{VolodinKuznetsovFomenko}, who used them to formulate an algorithm that they conjectured would detect whether or not a given Heegaard diagram presents $S^3$. Their key observation was the following \cite[Theorem 4.3.1]{VolodinKuznetsovFomenko}:

%\begin{figure}[h]
%\includegraphics[width=2in]{wave}
%\put(-90,0){\textcolor{red}{$\alpha$}}
%\put(-105,145){\textcolor{red}{$\alpha$}}
%\put(-90,100){$\gamma$}
%\caption{A wave.}
%\end{figure}

\begin{figure}
\labellist
 \pinlabel {{\color{red} $\alpha_i$}} [br] at 40 228
 \pinlabel {{\color{red} $\alpha_i$}} [tr] at 40 36
 \pinlabel {{\color{purple} $\alpha_i^{1}$}} [Br] at 91 170
 \pinlabel {{\color{green} $\alpha_i^{2}$}} [Bl] at 101 170
 \pinlabel $\gamma$ [tr] at 95 18
 \pinlabel {{\color{red} $\alpha_i$}} [tr] at 348 228
 \pinlabel {{\color{red} $\alpha_i$}} [tr] at 348 36
 \pinlabel {{\color{purple} $\alpha_i^{1}$}} [Br] at 399 170
 \pinlabel {{\color{green} $\alpha_i^{2}$}} [Bl] at 409 170
 \pinlabel $\gamma$ [tr] at 403 18
\endlabellist
\includegraphics[width=5in]{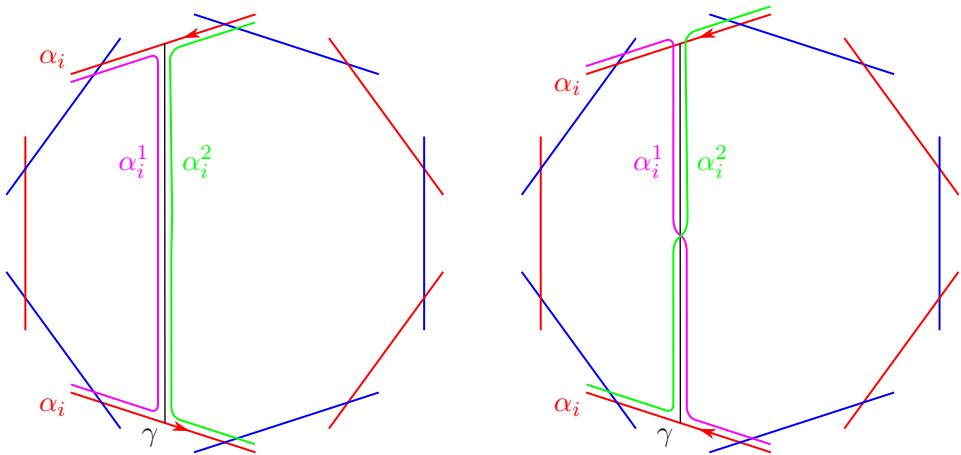}
\caption{A wave (left) and an anti-wave (right), along with the curves $\alpha_i^{1}$ and $\alpha_i^{2}$, perturbed to be transverse to $\alpha_i$. For the wave, the three curves $\alpha_i$, $\alpha_i^{1}$, $\alpha_i^{2}$ are pairwise disjoint and cobound a pair of pants; for the antiwave, each pair of curves meets in a point.}
\end{figure}

\begin{lemma}\label{l: wave swap}
Suppose that $\gamma$ is a wave in a Heegaard diagram $H = (S,\alpha,\beta)$ with endpoints on the curve $\alpha_i$. Let $a_1$ and $a_2$ denote the two arcs into which $\del \gamma$ splits $\alpha_i$, and set $\alpha_i^{1} = a_1 \cup \gamma$ and $\alpha_i^{2} = a_2 \cup \gamma$.\footnote{Observe that the curves $\alpha_i^{1}$ and $\alpha_i^{2}$ can be isotoped slightly to cobound a pair of pants in $S$ with $\alpha_i$.} Then it is possible to replace $\alpha_i$ with exactly one of $\alpha_i^{1}$ and $\alpha_i^{2}$ to obtain a Heegaard diagram $H' = (S, \alpha', \beta)$ which presents the same manifold as $H$. An analogous statement holds with the roles of $\alpha$ and $\beta$ exchanged. \qed
\end{lemma}

%\begin{proof}
%Let $U$ denote the handlebody specified by $\underline{\alpha}$, and surger $U$ along the compressing disks specified by these curves.  The result is a ball $B$ with a collection of $g$ pairs of disks on its boundary, from which we recover $U$ by appropriately regluing the disks in pairs.  Let $D', D''$ denote the disks that result from cutting open $\alpha$.  Since $\gamma$ meets $\alpha$ from the same side, it lifts to a curve in $\del B$ that starts and ends on $\del D'$, say.  Regard $\alpha_1$ and $\alpha_2$ as curves in $\del B$.  The curve $\alpha_i$ decomposes $\del B$ into a pair of disks, and we denote by $D_i$ the one disjoint from the interior of $D'$.   Thus, $\del B = D' \cup D_1 \cup D_2$, and the interiors of these disks are disjoint.  It follows that some $D_i$ contains $D''$ in its interior.  Then, up to sign, $[\alpha_i] = [\alpha]$ in $H_1(S;\bZ)$, modulo the classes represented by the other curves in $\underline{\alpha}$.  It follows that replacing $\alpha$ with $\alpha_i$ results in a collection of curves $\underline{\alpha'}$ whose classes are linearly independent in $H_1(S;\bZ)$ and which bound compressing disks in $U$ (pushing $int(D_i)$ into $U$ gives a compressing disk with boundary $\alpha_i$).  It follows that the collection $\underline{\alpha'}$ specifies the handlebody $U$, so the Heegaard diagram $H' = (S,\underline{\alpha'},\underline{\beta})$ presents the same manifold as $H$.
%%
%\end{proof}

Given an arbitrary Heegaard diagram $H$, iterating the {\em wave move} described in Lemma \ref{l: wave swap} yields a Heegaard diagram for $Y$ containing no waves. Volodin, Kuznetsov, and Fomenko conjectured that the only waveless diagram for $S^3$ in any given genus $g$ is the standard one.  This conjecture was subsequently shown to be true for Heegaard diagrams of genus $2$ \cite{HommaOchiaiTakahashi} but false for diagrams of higher genus \cite{MorikawaCounterexample, OchiaiCounterexample, ViroKopelski}.

\begin{proposition}
\label{p: wave}
A 1-extendible, strong Heegaard diagram $H$ for a strong L-space $Y$ does not contain a wave.
\end{proposition}
\begin{proof}
Without loss of generality, suppose to the contrary that $H$ contains a wave $\gamma$ with endpoints on the curve $\alpha_i$. Consider the curves $\alpha_i^{1}$, $\alpha_i^{2}$ defined in Lemma \ref{l: wave swap}. Observe that $\alpha_i \cap \beta$ is the disjoint union of the nonempty sets $\alpha_i^{1} \cap \beta$ and $\alpha_i^{2} \cap \beta$. Since $H$ is $1$-extendible, $\SS(H)$ is a nontrivial disjoint union $\SS_1 \sqcup \SS_2$, where $\SS_j$ consists of the generators that use a point in $\alpha_i^{j}$.  Let $H'$ denote the Heegaard diagram for $Y$ guaranteed by Lemma \ref{l: wave swap} with, say, $\alpha'_i = \alpha_i^{j}$. Note that $H'$ is strong, and $\SS(H') = \SS_j$. Then
\[
\det(Y) = \abs{\SS(H')}  = \abs{\SS_j} < \abs{\SS(H)} = \det(Y),
\]
a contradiction.
\end{proof}

\subsection{Anti-waves and formal L-spaces}

The set $\cS$ of strong L-spaces constitutes a class of 3-manifolds for which it is easy to certify the property that they are L-spaces.  The set of formal L-spaces constitutes another such class.  This family is well-known to experts but has not previously appeared in the literature.  To define it, recall that a {\em triad} is a triple of closed, oriented 3-manifolds that are obtained by Dehn filling a compact manifold with torus boundary along a triple of curves at pairwise distance (i.e., geometric intersection number) one.  A simple and useful way to recognize a triad is from a triple of Heegaard diagrams $H_i=(S,\alpha_0 \cup \alpha_i,\beta)$, $i=1,2,3$, where $\alpha_1, \alpha_2, \alpha_3$ are three curves on $S$ at pairwise distance one.  If $Y_i$ denotes the manifold presented by $H_i$, then $(Y_1,Y_2,Y_3)$ forms a triad.  In general, if $(Y_1,Y_2,Y_3)$ denotes a triad of rational homology spheres, then an elementary calculation in homology shows that they may be permuted so that $\det(Y_1) + \det(Y_2) = \det(Y_3)$.
\begin{definition}
\label{def: formal}
The set of {\em formal L-spaces} $\cF$ is the smallest set of manifolds such that
\begin{itemize}
\item
$S^3 \in \cF$, and
\item
if $(Y_1,Y_2,Y_3)$ is a triad with $Y_1,Y_2 \in \cF$ and $\det(Y_1)+\det(Y_2) = \det(Y_3)$, then $Y_3 \in \cF$.
\end{itemize}
\end{definition}
An elementary application of the surgery triangle in Heegaard Floer homology shows that every manifold in $\cF$ is an L-space.  Compare Definition \ref{def: formal} to the definition of a quasi-alternating link \cite[Section 2]{OSzDouble}.  In particular, the branched double cover of a quasi-alternating link is a formal L-space. Both $\cS$ and $\cF$ are proper subsets of the set of all L-spaces, as the Poincar\'e homology sphere is in the complement of both.

\begin{proposition} \label{p: formal-strong}
There exists a formal L-space that is not a strong L-space; thus, $\cF \not\subset \cS$.
\end{proposition}

%
%\begin{proof}
%The pretzel link $L = P(2,2,-3)$ is a non-alternating, quasi-alternating link with determinant 8.  Its branched double cover $Y=\Sigma(L)$ is the small Seifert fibered space $\Sigma(2,2,-3)$, and it belongs to $\cL$ as remarked above.  By the classification of involutions on small Seifert fibered spaces, it follows that no other knot in $S^3$ has branched double cover $Y$ (cf. \cite[Introduction]{mecchiareni2002}).  In particular, $Y$ is not the branched double-cover of an alternating link.  By Theorem \ref{t: le8}, it follows that $Y$ is not a strong L-space.  Thus, $Y \in \cL - \cS$.
%\end{proof}

\begin{proof}
The pretzel link $L = P(2,2,-3)$ is a non-alternating, quasi-alternating link with determinant $8$.  Its branched double cover $Y=\Sigma(L)$ is the small Seifert fibered space $\Sigma(2,2,-3)$, and it belongs to $\cF$ as remarked above.  Since $\det(Y)=8$ and $Y$ is not a connected sum of lens spaces, Theorem \ref{t: le8} shows that $Y$ is not a strong L-space. Thus, $Y \in \cF - \cS$.
\end{proof}

\begin{question} \label{q: strong-formal}
Is every strong L-space a formal L-space? That is, do we have $\cS \subset \cF$?
\end{question}

Note that an affirmative answer to Question \ref{q: alternating} would imply an affirmative answer to Question \ref{q: strong-formal}. To approach Question \ref{q: strong-formal}, we introduce the notion of an \emph{anti-wave} in a Heegaard diagram $H$.  This is an arc $\gamma$ that is properly embedded in a region $R$ of $H$ and whose endpoints lie on the interiors of distinct arcs of $\del R$ and on the same $\alpha$ or $\beta$ curve, such that the local signs of intersection at the two endpoints of $\gamma$ are the same. The relevance of anti-waves is given by the following:

\begin{proposition}
\label{p: anti-wave}
If a 1-extendible strong Heegaard diagram $H$ for a strong L-space $Y$ contains an anti-wave, then $Y$ fits into a triad with strong L-spaces $Y_1$ and $Y_2$ with $\det(Y) = \det(Y_1) + \det(Y_2)$.
\end{proposition}

Thus, one might try to answer Question \ref{q: strong-formal} by showing that every strong L-space admits a strong, 1-extendible Heegaard diagram that contains an anti-wave.

\begin{proof}[Proof of Proposition \ref{p: anti-wave}]
Without loss of generality, assume that $H$ contains an anti-wave $\gamma$ with endpoints on $\alpha_i$. As in the case of a wave, let $a_1$, $a_2$ denote the two arcs of $\alpha_i \minus \partial \gamma$, and let $\alpha_i^{j} = a_j \cup \gamma$ for $j=1,2$, with orientation induced from that of $\alpha_i$. We may perturb $\alpha_i^{1}$ and $\alpha_i^{2}$ so that they meet $\alpha_i$ and each other transversally, with each pair of curves meeting in a single point.
%We may index $\alpha_i^{1}$ and $\alpha_i^{2}$ and orient the curves such that
%\[
%\#(\alpha_i \cap \alpha_i^{1}) = \#(\alpha_i^{1} \cap \alpha_i^{2}) = \#(\alpha_i^{2} \cap \alpha_i) = -1.
%\]
Just as in the proof of Proposition \ref{p: wave}, there is a natural identification of $\alpha_i \cap \beta$ with the disjoint union $(\alpha_i^{1} \cap \beta) \sqcup (\alpha_i^{2} \cap \beta)$, leading to a nontrivial decomposition
$\SS(H) = \SS_1 \sqcup \SS_2$.

For $j=1,2$, let $H_j=(S,\alpha^{j}, \beta)$ be the diagram obtained by replacing $\alpha_i$ with the perturbed copy of $\alpha_i^{j}$. Note that $\SS(H_j)$ is identified with $\SS_j$, preserving all signs of intersection. We must verify that $H_j$ is a Heegaard diagram, i.e., that the curves in $\alpha^{j}$ are linearly independent in $H_1(S;\Z)$. Since all points in $\SS(H)$ have the same sign, the intersection matrix $M(H_j)$ has
\[
\abs{\det(M(H_j))} = \abs{\SS(H_j)} = \abs{\SS_j} \ne 0.
\]
Therefore, the curves in $\alpha^{j}$ must be linearly independent, as required. Let $Y_j$ be the $3$-manifold presented by $H_j$; then $Y_j$ is a strong L-space. By construction, $(Y_1,Y_2,Y)$ forms a triad.  Moreover,
\[
\det(Y) = \abs{\SS(H)} = \abs{\SS_1} + \abs{\SS_2} = \abs{\SS(H_1)} + \abs{\SS(H_2)} = \det(Y_1) + \det(Y_2),
\]
as required.
%Observe that the intersection matrices $M(H)$, $M(H_1)$, and $M(H_2)$ differ only in their $i\Th$ rows, and the sum of the $i\Th$ rows of $M(H_1)$ and $M(H_2)$ equals the $i\Th$ row of $M(H)$. Since all three of these matrices are P\'olya matrices, it follows that
%\[
%\det(Y) = \per(\abs{M(H)}) = \per(\abs{M(H_1)}) + \per(\abs{M(H_2)}) = \det(Y_1) + \det(Y_2),
%\]
%as required.
%check that very last part
\end{proof}

%%%
%%%
%%%

\subsection{Weak reducibility}
\label{subsec: reducibility}

A Heegaard splitting is called {\em weakly reducible} if there exist disjoint compressing disks in the two handlebodies into which the Heegaard surface decomposes the 3-manifold. This notion was introduced by Casson and Gordon, who showed that given a weakly reducible Heegaard splitting of a manifold, either the splitting is reducible or else the manifold contains an incompressible surface of positive genus \cite[Theorem 3]{cg:heegaard}.  The following result is elementary.

\begin{proposition} \label{p: weakly reducible}
A strong Heegaard diagram of genus $g \ge 3$ describes a weakly reducible Heegaard splitting.
\end{proposition}

\begin{proof}
Given a strong Heegaard diagram, convert it into a 1-extendible strong Heegaard diagram $H$ for the same Heegaard splitting by Proposition \ref{p: 1-extendible}.  Then $M=M(H)$ is a P\'olya matrix and $H$ is coherent.  The proposition then follows from the assertion that a $g \times g$ P\'olya matrix $M$ contains a zero entry for $g \ge 3$, which we now establish.  Let
\[
N = \begin{pmatrix}
b&o&t \\
c&u&d \\
r&a&g
\end{pmatrix}
\]
denote the top-left $3 \times 3$ minor of $M$ and $m$ the product of its remaining diagonal entries.  Since $M$ is a P\'olya matrix, either $m=0$ or else all nonzero terms in the expansion
\[
\det(N) = bug + cat + rod - bad - cog - rut
\]
have the same sign, since each of these terms' products with $m$ contributes with the same sign to $\det(M)$. It follows that $0 \le dogcartbum \le 0$, so $M$ has at least one $0$ entry.
\end{proof}

In fact, for a strong Heegaard diagram of genus greater than $3$, \cite[Corollary 7.8]{rst:polya} immediately implies a much stronger result:

\begin{proposition} \label{p: atmost3}
A strong, 1-extendible Heegaard diagram contains an $\alpha$ curve that meets at most three $\beta$ curves, and vice versa. \qed
\end{proposition}

It is possible that Propositions \ref{p: weakly reducible} and \ref{p: atmost3} could be useful towards the classification of strong, 1-extendible Heegaard diagrams of higher genus.

\section{Questions for future research} \label{sec: questions}

We conclude with a compilation of questions that may interest other researchers, particularly those with expertise in the theory of Heegaard splittings. (See also Questions \ref{q: stronggenus}, \ref{q: graph}, \ref{q: simpleknot}, and \ref{q: strong-formal}, above.)

\begin{question} \label{q: strongmoves1}
\textup{What can be said about the set of strong Heegaard diagrams for a given strong L-space? Is there a finite collection of moves on strong diagrams that can interpolate between any two given strong diagrams for the same strong L-space, and if so, can we arrange that the genus change monotonically in such sequences? As a specific example, the branched double cover of a two-bridge link is a lens space, and one can try to relate the large-genus strong diagram produced by the first author \cite{GreeneSpanning} to the standard genus-$1$ diagram.}
\end{question}

\begin{question} \label{q: strongmoves2}
\textup{More generally, is there a finite collection of moves on Heegaard diagrams that can interpolate between any two given diagrams for the same 3-manifold such that $\abs{\SS(H)}$ remains bounded in some way? Such a collection might yield an algorithm that can detect whether a given Heegaard diagram presents a strong L-space. Past efforts to study 3-manifolds algorithmically via Heegaard diagrams (e.g. \cite{VolodinKuznetsovFomenko}) have focused on genus and the total number of intersection points in $\alpha \cap \beta$ as the measures of complexity that should be bounded.  We wonder if $\abs{\SS(H)}$ might be more useful for this purpose.}
\end{question}

\begin{question} \label{q: summands}
\textup{Is the simultaneous trajectory number multiplicative under connected sum?  A theorem of Haken implies that any Heegaard diagram for a reducible manifold can be transformed into a reducible diagram using isotopies and handleslides \cite{HakenSurfaces}.  Can this be accomplished without increasing the number of generators?  As a derivative of that question, are all summands of a strong L-space themselves strong L-spaces?  The K\"unneth formula for Heegaard Floer homology implies that the summands of a reducible L-space are themselves L-spaces.
%The K\"unneth formula for Heegaard Floer homology implies that if $Y_1 \conn Y_2$ is an L-space, then $Y_1$ and $Y_2$ are L-spaces.
%If $M$ were additive, then it would follow that a reducible strong L-space has strong L-space summands.
%Is the same true for strong L-spaces? That is, if $Y$ is a strong L-space, must $Y_1$ and $Y_2$ be strong L-spaces?
}
\end{question}

\begin{question} \label{q: foliation}
\textup{
Is there an elementary proof that a strong L-space does not admit a co-orientable taut foliation?  Compare the discussion following Definition \ref{d: SLS}.
%The work of Lewallen and the second author \cite{LevineLewallen} gives an elementary proof that a strong L-space cannot admit an $\R$-covered taut foliation, without using the full machinery of Heegaard Floer homology (specifically, Ozsv\'ath and Szab\'o's result that an L-space cannot not admit a taut, co-orientable foliation \cite{OSzGenus}).
}
\end{question}

%\begin{question}
%\textup{Does a strong 1-extendible Heegaard diagram contain an anti-wave?}
%\end{question}

\begin{question}
\textup{If $H$ is a Heegaard diagram for a rational homology sphere and the differential vanishes on $\CF(H)$, does it follow that $H$ is a strong Heegaard diagram? Note that if the differential vanishes for some choice of analytic data, then it must vanish for all such choices.  However, the fact that it vanishes may rely on some non-trivial analysis of moduli spaces. A related question is whether a Heegaard diagram in which there are no non-negative domains of Whitney disks must be strong.  This assumption is stronger but purely combinatorial.}
\end{question}

\bibliographystyle{myalpha}
\bibliography{bibliography}

\def\cprime{$'$}
\providecommand{\bysame}{\leavevmode\hbox to3em{\hrulefill}\thinspace}
\providecommand{\MR}{\relax\ifhmode\unskip\space\fi MR }
\providecommand{\MRhref}[2]{%
  \href{http://www.ams.org/mathscinet-getitem?mr=#1}{#2}
}
\providecommand{\href}[2]{#2}
\begin{thebibliography}{BGW13}

\bibitem[Ban30]{BankwitzAlternating}
Carl Bankwitz, \emph{\"{U}ber die {T}orsionszahlen der alternierenden
  {K}noten}, Math. Ann. \textbf{103} (1930), no.~1, 145--161.

\bibitem[Ber90]{berge:lens}
John Berge, \emph{Some knots with surgeries yielding lens spaces}, Unpublished
  manuscript (c. 1990).

\bibitem[BGW13]{BoyerGordonWatson}
Steven Boyer, Cameron~McA. Gordon, and Liam Watson, \emph{On {L}-spaces and
  left-orderable fundamental groups}, Math. Ann. \textbf{356} (2013), no.~4,
  1213--1245.

\bibitem[BW69]{BredonWood}
Glen~E. Bredon and John~W. Wood, \emph{Non-orientable surfaces in orientable
  {$3$}-manifolds}, Invent. Math. \textbf{7} (1969), 83--110.

\bibitem[CG87]{cg:heegaard}
Andrew~J. Casson and Cameron~McA. Gordon, \emph{Reducing {H}eegaard
  splittings}, Topology Appl. \textbf{27} (1987), no.~3, 275--283.

\bibitem[Cro59]{CrowellNonalternating}
Richard~H. Crowell, \emph{Nonalternating links}, Illinois J. Math. \textbf{3}
  (1959), 101--120.

\bibitem[Cro04]{Cromwell:book}
Peter Cromwell, \emph{Knots and links}, Cambridge University Press, Cambridge,
  2004.

\bibitem[DT03]{DunfieldThurstonExperiements}
Nathan~M. Dunfield and William~P. Thurston, \emph{The virtual {H}aken
  conjecture: experiments and examples}, Geom. Topol. \textbf{7} (2003),
  399--441.

\bibitem[Gor09]{GordonPCMI}
Cameron~McA. Gordon, \emph{Dehn surgery and 3-manifolds}, Low dimensional
  topology, IAS/Park City Math. Ser., vol.~15, Amer. Math. Soc., Providence,
  RI, 2009, pp.~21--71.

\bibitem[Gre13a]{GreeneLattices}
Joshua~E. Greene, \emph{Lattices, graphs, and {C}onway mutation}, Invent. Math.
  \textbf{192} (2013), no.~3, 717--750.

\bibitem[Gre13b]{GreeneSpanning}
\bysame, \emph{A spanning tree model for the {H}eegaard {F}loer homology of a
  branched double-cover}, J. Topol. \textbf{6} (2013), no.~2, 525--567.

\bibitem[GW13]{GreeneWatsonQA}
Joshua~E. Greene and Liam Watson, \emph{Turaev torsion, definite 4-manifolds,
  and quasi-alternating knots}, Bull. Lond. Math. Soc. \textbf{45} (2013),
  no.~5, 962--972.

\bibitem[Hak68]{HakenSurfaces}
Wolfgang Haken, \emph{Some results on surfaces in {$3$}-manifolds}, Studies in
  {M}odern {T}opology, Math. Assoc. Amer. (distributed by Prentice-Hall,
  Englewood Cliffs, N.J.), 1968, pp.~39--98.

\bibitem[Hed11]{HeddenBerge}
Matthew Hedden, \emph{On {F}loer homology and the {B}erge conjecture on knots
  admitting lens space surgeries}, Trans. Amer. Math. Soc. \textbf{363} (2011),
  no.~2, 949--968.

\bibitem[Het64]{hetyei}
G\'abor Hetyei, \emph{Rectangular configurations which can be covered by 2
  $\times$ 1 rectangles}, P\'ecis Tan. F{\H o}isk. K{\H o}zl. \textbf{8}
  (1964), no.~5, 351--367.

\bibitem[HOT80]{HommaOchiaiTakahashi}
Tatsuo Homma, Mitsuyuki Ochiai, and Moto-o Takahashi, \emph{An algorithm for
  recognizing {$S^3$} in {$3$}-manifolds with {H}eegaard splittings of genus
  two}, Osaka J. Math. \textbf{17} (1980), 625--648.

\bibitem[KR14]{KazezRoberts}
William~H. Kazez and Rachel Roberts, \emph{Approximating {$C^0$} foliations},
  \arxiv{1404.5919}, 2014.

\bibitem[LL08]{LeeLipshitz}
Dan~A. Lee and Robert Lipshitz, \emph{Covering spaces and
  {$\mathbf{Q}$}-gradings on {H}eegaard {F}loer homology}, J. Symplectic Geom.
  \textbf{6} (2008), no.~1, 33--59.

\bibitem[LL12]{LevineLewallen}
Adam~S. Levine and Sam Lewallen, \emph{Strong {L}-spaces and
  left-orderability}, Math. Res. Lett. \textbf{19} (2012), no.~6, 1237--1244.

\bibitem[LP09]{LovaszPlummerBook}
L{\'a}szl{\'o} Lov{\'a}sz and Michael~D. Plummer, \emph{Matching theory}, AMS
  Chelsea Publishing, Providence, RI, 2009, corrected reprint of the 1986
  original.

\bibitem[LRS13]{LevineRubermanStrleNonorientable}
Adam~S. Levine, Daniel Ruberman, and Sa\v{s}o Strle, \emph{Nonorientable
  surfaces in homology cobordisms}, Geom. Topol. (2013), with an appendix by
  Ira M. Gessel, to appear, arXiv:1310.8516.

\bibitem[McC04]{mccuaig:polya}
William McCuaig, \emph{P\'olya's permanent problem}, Electron. J. Combin.
  \textbf{11} (2004), no.~1, Research Paper 79, 83 pp. (electronic).

\bibitem[Mor80]{MorikawaCounterexample}
Osamu Morikawa, \emph{A counterexample to a conjecture of {W}hitehead}, Math.
  Sem. Notes Kobe Univ. \textbf{8} (1980), no.~2, 295--298.

\bibitem[MPR14]{MartelliPetronioRoukema}
Bruno Martelli, Carlo Petronio, and Fionntan Roukema, \emph{Exceptional {D}ehn
  surgery on the minimally twisted five-chain link}, Comm. Anal. Geom.
  \textbf{22} (2014), 689--735.

\bibitem[MT91]{mt:tait}
William Menasco and Morwen~B. Thistlethwaite, \emph{The {T}ait flyping
  conjecture}, Bull. Amer. Math. Soc. (N.S.) \textbf{25} (1991), no.~2,
  403--412.

\bibitem[NW14]{NiWuRational}
Yi~Ni and Zhongtao Wu, \emph{Heegaard {F}loer correction terms and rational
  genus bounds}, Adv. Math. \textbf{267} (2014), 360--380.

\bibitem[Och79]{OchiaiCounterexample}
Mitsuyuki Ochiai, \emph{A counterexample to a conjecture of {W}hitehead and
  {V}olodin-{K}uznetsov-{F}omenko}, J. Math. Soc. Japan \textbf{31} (1979),
  no.~4, 687--691.

\bibitem[OS03]{OSzAbsolute}
Peter Ozsv{\'a}th and Zolt{\'a}n Szab{\'o}, \emph{Absolutely graded {F}loer
  homologies and intersection forms for four-manifolds with boundary}, Adv.
  Math. \textbf{173} (2003), no.~2, 179--261.

\bibitem[OS04a]{OSzGenus}
\bysame, \emph{Holomorphic disks and genus bounds}, Geom. Topol. \textbf{8}
  (2004), 311--334 (electronic).

\bibitem[OS04b]{OSzKnot}
\bysame, \emph{Holomorphic disks and knot invariants}, Adv. Math. \textbf{186}
  (2004), no.~1, 58--116.

\bibitem[OS04c]{OSzProperties}
\bysame, \emph{Holomorphic disks and three-manifold invariants: properties and
  applications}, Ann. of Math. (2) \textbf{159} (2004), no.~3, 1159--1245.

\bibitem[OS04d]{OSz3Manifold}
\bysame, \emph{Holomorphic disks and topological invariants for closed
  three-manifolds}, Ann. of Math. (2) \textbf{159} (2004), no.~3, 1027--1158.

\bibitem[OS05]{OSzDouble}
\bysame, \emph{On the {H}eegaard {F}loer homology of branched double-covers},
  Adv. Math. \textbf{194} (2005), no.~1, 1--33.

\bibitem[OS06]{OSz4Manifold}
\bysame, \emph{Holomorphic triangles and invariants for smooth four-manifolds},
  Adv. Math. \textbf{202} (2006), no.~2, 326--400.

\bibitem[Ras03]{RasmussenThesis}
Jacob~A. Rasmussen, \emph{Floer homology and knot complements}, Ph.D. thesis,
  Harvard University, 2003, \arxiv{math/0509499}.

\bibitem[Ras07]{RasmussenLens}
\bysame, \emph{Lens space surgeries and {L}-space homology spheres},
  \arxiv{0710.2531}, 2007.

\bibitem[RST99]{rst:polya}
Neil Robertson, P.~D. Seymour, and Robin Thomas, \emph{Permanents, {P}faffian
  orientations, and even directed circuits}, Ann. of Math. (2) \textbf{150}
  (1999), no.~3, 929--975.

\bibitem[Sch93]{schrijver:tait}
Alexander Schrijver, \emph{Tait's flyping conjecture for well-connected links},
  J. Combin. Theory Ser. B \textbf{58} (1993), no.~1, 65--146.

\bibitem[Usu12a]{UsuiSmoothing1}
Takuya Usui, \emph{Heegaard {Floer} homology, {L}-spaces, and smoothing order
  on links {I}}, \arxiv{1202.1353}, 2012.

\bibitem[Usu12b]{UsuiSmoothing2}
\bysame, \emph{Heegaard {Floer} homology, {L}-spaces, and smoothing order on
  links {II}}, \arxiv{1202.3333}, 2012.

\bibitem[VK77]{ViroKopelski}
O.~Ja. Viro and V.~L. Kobel{\cprime}ski{\u\i}, \emph{The
  {V}olodin-{K}uznecov-{F}omenko conjecture on {H}eegaard diagrams is false},
  Uspehi Mat. Nauk \textbf{32} (1977), no.~5(197), 175--176.

\bibitem[VKF74]{VolodinKuznetsovFomenko}
I.~A. Volodin, V.~E. Kuznetsov, and A.~T. Fomenko, \emph{The problem of
  discriminating algorithmically the standard three-dimensional sphere}, Russ.
  Math. Surveys \textbf{29} (1974), no.~5, 71--172.

\bibitem[Voo79]{VoorhoevePermanents}
Marc Voorhoeve, \emph{A lower bound for the permanents of certain
  {$(0,\,1)$}-matrices}, Nederl. Akad. Wetensch. Indag. Math. \textbf{41}
  (1979), no.~1, 83--86.

\bibitem[VY89]{vy:pfaffian}
Vijay~V. Vazirani and Milhalis Yannakakis, \emph{Pfaffian orientations,
  {$0$}-{$1$} permanents, and even cycles in directed graphs}, Discrete Appl.
  Math. \textbf{25} (1989), no.~1-2, 179--190, Combinatorics and complexity
  (Chicago, IL, 1987).

\end{thebibliography}

\end{document}